\documentclass[a4paper]{amsart}
\usepackage{subfigure}
\usepackage{graphicx}
\usepackage{color}
\usepackage{amssymb}
\usepackage[all]{xy}
\usepackage{amsmath}
\usepackage{epic}
\usepackage{subfigure}
\usepackage{epic}
\usepackage{eepic}
\usepackage{setspace}
\usepackage{booktabs}
\usepackage{longtable}

\newcommand{\D}{\mathcal{D}}
\newcommand{\sigmav}{\check{\sigma}}
\newcommand{\deltav}{\check{\delta}}

\newcommand{\face}[2]{\text{face}\left(#1,#2\right)}

\newtheorem{thm}{Theorem}[section]
\newtheorem{lem}[thm]{Lemma}
\newtheorem{prop}[thm]{Proposition}
\newtheorem{cor}[thm]{Corollary}
\theoremstyle{definition}
\newtheorem{defn}[thm]{Definition}
\newtheorem{exmp}[thm]{Example}
\newtheorem{rem}[thm]{Remark}

\newcommand{\CC}{\mathbf{C}}
\newcommand{\RR}{\mathbf{R}}
\newcommand{\QQ}{\mathbf{Q}}
\newcommand{\ZZ}{\mathbf{Z}}
\newcommand{\NN}{\mathbf{N}}
\newcommand{\PP}{\mathbf{P}}
\newcommand{\A}{\mathbf{A}}
\newcommand{\m}{\mathfrak{m}}

\DeclareMathOperator{\relint}{relint}
\DeclareMathOperator{\tail}{tail}
\DeclareMathOperator{\SF}{SF}
\DeclareMathOperator{\tcadiv}{T-CaDiv}

\DeclareMathOperator{\Hom}{Hom}
\DeclareMathOperator{\divisor}{div}

\DeclareMathOperator{\wdiv}{Div}

\DeclareMathOperator{\spec}{Spec}

\DeclareMathOperator{\ord}{ord}
\DeclareMathOperator{\supp}{supp}
\DeclareMathOperator{\xrays}{x-rays}
\DeclareMathOperator{\Pol}{Pol}
\DeclareMathOperator{\loc}{Loc}
\DeclareMathOperator{\pdv}{X}

\DeclareMathOperator{\conv}{conv}
\DeclareMathOperator{\pos}{pos}
\DeclareMathOperator{\vol}{vol}
\DeclareMathOperator{\Graph}{\Gamma}
\newcommand{\fan}{\Xi}

\renewcommand{\O}{{\mathcal{O}}}
\newcommand{\CO}{{\mathcal{O}}}

\newcommand{\X}{{\mathfrak{X}}}

\newcommand{\tropical}[1]{{\mathcode`\*="020C\mathcode`\+="0208\renewcommand{\cdot}{\odot}\renewcommand{\sum}{\bigoplus}
$#1$}}
\newcommand{\mtrop}[1]{\text{\tropical{#1}}}

\begin{document}
\title{Canonical divisors on T-varieties}
%
%
\author[H.~S\"uss]{Hendrik S\"uss }
\address{Institut f\"ur Mathematik,
        LS Algebra und Geometrie,
        Brandenburgische Technische Universit\"at Cottbus,
        PF 10 13 44, 
        03013 Cottbus, Germany}
\email{suess@math.tu-cottbus.de}

\begin{abstract}
  Generalising toric geometry we study compact varieties admitting lower dimensional torus actions. In particular we describe divisors on them in terms of convex geometry and give a criterion for their ampleness. These results may be used to study Fano varieties with small torus actions. As a first result we classify log del Pezzo $\CC^*$-surfaces of Picard number 1 and Gorenstein index $\leq 3$. In further examples we show how classification might work in higher dimensions and we give explicit descriptions of some equivariant smoothings of Fano threefolds.
\end{abstract}

\maketitle
\tableofcontents
\section*{Introduction}
A lot of effort has been spend on the classification of Fano varieties. In dimension $2$ singular del Pezzo surfaces are completely classified up to Gorenstein index $2$. In dimension~$3$ there exists a complete classification of smooth Fano varieties (\cite{0382.14013}, \cite{0424.14012}, \cite{MR641971}).

Toric geometry has been successfully used to study Fano varieties and their applications in mirror symmetry (\cite{MR1269718}, \cite{MR1463173}).
There exists a classification of Gorenstein toric Fano varieties up to dimension $5$ and the conjectured boundedness of Fano varieties was proven for toric varieties \cite{1468-4802-75-1-A17}. 

On the other hand toric geometry handles only a very restricted class of varieties. In order to extend toric results and techniques to a broader class of varieties we may weaken our preconditions and consider T-varieties, i.e. we require only the effective action of a lower dimensional torus.  Not much analysis has been done for this special case, yet. Hence, there are only general statements available.

 In order to study Fano varieties we suggest to use the description of T-varieties from \cite{divfans} and invariant divisors on them from \cite{petersen-suess08}. For the case that dimensions of the variety and the dimension of the acting torus differ by one these descriptions coincide with that given by Timashev \cite{0911.14022,1034.14004}.

 For classification purposes criteria for ampleness as well as formul\ae\ for important invariants as the Gorenstein index, the Fano index, the Picard number and the Fano degree are needed. In this paper we provide these tools and give first applications.

\goodbreak

The paper is organised as follows:
In section~\ref{sec:t-varieties} we recall the description of T-varieties 
by polyhedral divisors and divisorial fans concentrating on the special case of torus actions of codimension~1.

In section~\ref{sec:invariant-divisors} we study invariant divisors, describe them by piecewise affine functions and give a criterion for ampleness. In addition to that we provide a way to calculate a canonical divisor and intersection numbers.

Section~\ref{sec:singularities} uses these results to study local properties on T-varieties such as smoothness, $\QQ$-factoriallity, $\QQ$-Gorenstein and log-terminal properties. 

As a first application in section~\ref{sec:del-pezzo} we classify log del Pezzo $\CC^*$-surfaces of Picard rank $1$ and Gorenstein index $\leq 3$.

In section~\ref{sec:examples-dimension-3} we pass to dimension $3$ and give some ideas how classification might work here. In addition to that we give examples of equivariant smoothings of a Fano variety.

\section{T-Varieties}
\label{sec:t-varieties}
First we recall some facts and notations from convex geometry.
Here, $N$ always is a lattice and $M:=\Hom(N,\ZZ)$ its dual.
The associated $\QQ$-vector spaces $N \otimes \QQ$ and $M \otimes \QQ$ are denoted by $N_\QQ$ and $M_\QQ$ respectively.
Let $\sigma \subset N_\QQ$ be a pointed convex polyhedral cone.
A polyhedron $\Delta$ which can be written as a Minkowski sum 
$\Delta = \pi + \sigma$ of $\sigma$ and a compact polyhedron $\pi$ is said to have $\sigma$ as its tail cone.

With respect to Minkowski addition the polyhedra with tail cone $\sigma$ form a semi-group which we denote by $\Pol_{\sigma}^+(N)$.
Note that $\sigma \in \Pol_{\sigma}^+(N)$ is the neutral element of this semi-group and that $\emptyset$  by definition is also an element of $\Pol_{\sigma}^+(N)$.

A {\em polyhedral divisor} with tail cone $\sigma$ on a normal variety $Y$ is a 
formal finite sum
$$\D = \sum_D \Delta_D \otimes D,$$
where $D$ runs over all prime divisors on $Y$ and $\Delta_D \in \Pol^+_\sigma$.
Here, finite means that only finitely many coefficient differ from the 
tail cone, i.e. the neutral element of $\Pol^+_\sigma$.

We can evaluate a polyhedral divisor for every element $u \in \sigma^\vee \cap M$ via
$$\D(u):=\sum_D \min_{v \in \Delta_D} \langle u , v \rangle D,$$
in order to obtain an ordinary divisor on $\loc \D$.
Here, $\loc \D$ denotes the {\em locus} of $\D$, the subset of $Y$ where $\D$ has non-empty coefficients 
$\loc \D := Y \setminus \left( \bigcup_{\Delta_D = \emptyset} D \right)$.

\begin{defn}
  A polyhedral divisor $\D$ is called {\em proper} if 
  \begin{enumerate}
  \item it is Cartier, i.e. $\D(u)$ is Cartier for every $u \in \sigma^\vee \cap M$,
  \item it is semi-ample, i.e.  $\D(u)$ is semi-ample for every $u \in \sigma^\vee \cap M$,
  \item $\D$ is big outside the boundary, i.e. $\D(u)$ is big for every $u$ in the relative interior of $\sigma^\vee$.
  \end{enumerate}
\end{defn}

To a proper polyhedral divisor we associate an $M$-graded $\CC$-algebra and consequently
an affine scheme admitting a $T^N$-action:
$$\pdv(\D):= \spec \bigoplus_{u \in \sigma^\vee \cap M} \Gamma(\CO(\D(u))).$$

From \cite{MR2207875} we know that this construction gives a normal variety of dimension $\dim Y + \dim N$. Moreover, every normal affine variety with torus action can be obtained this way.

\begin{defn}
Let $D=\sum_D \Delta_D \otimes D$, $D'=\sum_D \Delta'_D \otimes D$ be two polyhedral divisors on $Y$.
\begin{enumerate}
\item We write $\D' \subset \D$ if $\Delta'_D \subset \Delta_D$ holds for every prime divisor $D$.
\item We define the {\em intersection} of two polyhedral divisors $$\D \cap \D' := \sum_D (\Delta'_D \cap \Delta_D) \otimes D.$$
\item We define the {\em degree} of a polyhedral divisor on a curve $Y$ to be  $$\deg \D := \sum_D \Delta_D \deg(D).$$
  {(\em Note: If $\D$ carries $\emptyset$-coefficients we automatically get $\deg \D = \emptyset$.)}
\item For a (not necessarily closed) point $y \in Y$ we define the {\em fibre polyhedron} $$\D_y := \sum_{y \in D} \Delta_D.$$
{(\em Note that we may recover $\Delta_D=\D_D$ this way.)}
\item We can restrict $\D$ to an open subset $U \subset Y$:
  $$\D|_U := \D + \sum_{D \cap U = \emptyset} \emptyset \otimes D$$
\end{enumerate}
\end{defn}

Assume $\D' \subset \D$. This implies
$$\bigoplus_{u \in \sigma^\vee \cap M} \Gamma( \CO(\D'(u)))
\supset \bigoplus_{u \in \sigma^\vee \cap M} \Gamma( \CO(\D(u)))$$ and we get a dominant morphism $\pdv(\D') \rightarrow \pdv(\D)$.

\begin{defn}
  If $\D' \subset \D$ and the corresponding morphism defines an open inclusion, then we say $\D'$ is a face of $\D$ and write $\D' \prec \D$.
\end{defn}

\begin{defn}\ 
  \begin{enumerate}
  \item   A {\em divisorial fan} is a finite set $\fan$ of polyhedral divisors such that for $\D,\D' \in \fan$ we have $\D \succ \D' \cap \D \prec \D'$.
  \item The polyhedral complex $\fan_y$ defined by the polyhedra $\D_y$ is called a {\em slice} of $\fan$.
  \item $\fan$ is called {\em complete} if all slices $\fan_y$ are complete subdivisions of $N_\QQ$.
  \end{enumerate}
\end{defn}

We may glue the affine varieties $\pdv(\D)$ via
$\pdv(\D) \hookleftarrow \pdv(\D \cap \D') \hookrightarrow \pdv(\D').$
By \cite[Prop. 5.3.]{divfans} we know that the cocycle condition is fulfilled, so we obtain a scheme this way.  In the case of a complete fan we get a complete variety.

As a divisorial fan $\fan$ corresponds to an open affine covering of $\pdv(\fan)$ it is not unique, because we may switch to another invariant open affine covering. We will do this occasionally by refining an existing divisorial fan.

Let us consider the affine case. $X = \spec A$, where $A = \bigoplus_{u} \Gamma(\CO(\D(u)))$. 
We have $A_0 = \Gamma(\loc \D,\CO_{\loc \D})$ and consequently get proper, surjective  maps to $Y_0 := \spec A_0$, the categorical quotient of $X$
$$q: X \rightarrow Y_0,\qquad \pi: \loc \D \rightarrow Y_0$$

\begin{lem}
\label{sec:lem-refine}
  Let $\D$ be a proper polyhedral divisor on $Y$ and $\{U_i\}_{i \in I}$ an open affine covering of $Y_0$. 
  Then $q^{-1}(U_i) \cong \pdv(\D|_{\pi^{-1}(U_i)})$. 
  Moreover we get a divisorial fan 
$\fan := \{\pdv(\D|_{\pi^{-1}(U_i)})\}_{i \in I}$
such that $\pdv(\fan) \cong \pdv(\D)$.
\end{lem}

\begin{proof}
  This is a direct consequence of proposition~3.3 in \cite{divfans}.
\end{proof}

\begin{rem}
\label{sec:rem-toric}
For classification purposes it is crucial to identify torus actions which extend to an effective action of $T^{\dim X}$, i.e. essentially toric situations.

  We may consider a complete toric variety $X:=X_\fan$ and restrict its torus action to that of a smaller torus $T \hookrightarrow T_X$. This gives rise to a divisorial fan in the following way. 

The embedding $T'\hookrightarrow T$ corresponds to an exact sequence of lattices $$0 \rightarrow N \stackrel{F}{\rightarrow} N_X \stackrel{P'}{\rightarrow} N' \rightarrow 0.$$ We may choose a splitting $N_X \cong N \oplus N'$ with projections $$P:N_X\rightarrow N,\quad P':N_X\rightarrow N'.$$

 Furthermore we choose $Y=\Sigma'$, where $\Sigma'$ is an arbitrary smooth projective fan refining $P(\Sigma)$. Then every cone $\sigma \in \Sigma$ yields a proper polyhedral divisor. For every ray $\rho \in \Sigma^{(1)}$ with primitive vector $n_\rho$ we set $\Delta_\rho = P(P'^{-1}(n_\rho))$ and $\D_\sigma=\sum_{\rho \in \Sigma^{(1)}} \Delta_\rho \otimes D_\rho$. Now $\{ \D_\sigma \}_{\sigma \in \Sigma}$ is a divisorial fan.

 The other way around we may consider a divisorial fan $\fan$ on a toric variety $Y=X_{\Sigma}$, such that $\D_D = \tail \D$ for every $\D \in \fan$ and every non-invariant prime divisor 
$D \subset Y$. Then every $\D = \sum_{\rho \in \Sigma^{(1)}} \Delta_\rho \otimes D_\rho \in \fan$ gives rise to a cone  $$\sigma_\D := \QQ^+ \cdot \left(\bigcup_{\rho \in \Sigma^{(1)}} \{n_\rho\}\times \Delta_{\rho} \right) \subset N' \oplus N.$$

Here the fact that the cone is pointed corresponds to the bigness of the corresponding polyhedral divisor and the convexity corresponds to the semi-ampleness.
\end{rem}

For the rest of the paper we restrict to the case that a torus of dimension $\dim X -1$  acts on $X$. This means that the underlying variety $Y$ of the corresponding divisorial fan is a projective curve. 

In this situation the locus of a polyhedral divisor may be affine or complete and we get simple criteria for properness and for the face relation:
\begin{itemize}
\item   $\D$ is a proper polyhedral divisor if $\deg \D \subsetneq \tail \D$ and for every $u \in \sigma^\vee$ with 
  $$\face{\tail \D}{u} \cap \deg \D \neq \emptyset$$
 some multiple of $\D(u)$ has to be principal.

\item 
  $\D' \prec \D$ holds if and only if 
  $\Delta'_P$ is a face of $\Delta_P$ for every point $P \in Y$ and we have $\deg \D \cap \tail \D' = \deg \D'$.
\end{itemize}

\begin{thm}[\cite{MR2207875}, Theorem~8.8.]
\label{sec:thm-iso-divisor}
  Let $\D,\D'$ be polyhedral divisors on $Y$, $Y'$. Then we have $\pdv(\D)\cong \pdv(\D')$ if and only if there are isomorphisms $\varphi:Y \rightarrow Y'$, $F:N'\rightarrow N$ such that
$$\D_{\varphi(P)}=F(\D'_P) + \sum_{i=1}^n a^i_P e_i,$$
  Here, $e_1,\ldots,e_n$ is a lattice basis of $N$ and $D_i := \sum_P a^i_P P$ are principal divisors.

  We write $\D \sim F_*\varphi^*\D'$ in this case.
\end{thm}

One could ask oneself if $\pdv(X)$ is already determined by prime divisor slices $\fan_D$ of $\fan$. In general this is not the case. We really need to know which polyhedra in different slices belong to the same polyhedral divisor.

For a divisorial fan on a curve which consists only of polyhedral divisors with affine locus the situation is different. If we consider two such fans $\fan,\fan'$ having the same slices lemma~\ref{sec:lem-refine} tells us that there exists a common refinement
$\fan''=\{\D|_U \mid \D \in \fan, U \in \mathcal{U}\} = \{\D|_U \mid \D \in \fan', U \in \mathcal{U}\}$ with $\mathcal{U}$ being a sufficient fine affine covering  of $Y$. We have $\pdv(\fan) \cong \pdv(\fan'') \cong \pdv(\fan')$. For $Y$ a complete curve we have aside from polyhedral divisors with affine locus only those with locus $Y$. For reconstructing $\pdv(\fan)$ from the slices we need to remember which of the maximal polyhedra came from divisors with complete loci. To empathise this fact we call maximal polyhedra coming from divisors with complete loci {\em marked}.

As a consequence of theorem~\ref{sec:thm-iso-divisor} we get the following theorem, which is essential for all classification purposes.

\begin{thm}
\label{sec:thm-fan-iso-codim1}
  Let $\fan, \fan'$ be a polyhedral divisors on a curve $Y$.
  We have $\pdv(\fan) \cong \pdv(\fan')$ if and only if there are
  isomorphisms $F:N' \rightarrow N$, $\varphi:Y\rightarrow Y'$.
  such that $F$ induces an isomorphism of marked subdivisions (i.e. respects marks):
  $$\fan_{\varphi(P)} = F(\fan'_P) + \sum_{i=1}^n a^i_P e_i,$$
  Here, $e_1,\ldots,e_n$ is a lattice basis of $N$ and $D_i := \sum_P a_P^i P$ are principal divisors.

\end{thm}

\begin{cor}
\label{sec:cor-toric}
  For a divisorial fan $\fan$ on a curve $Y$ the associated variety $\pdv(\fan)$ can be obtained from restricting the torus action of a toric variety (we say it is toric for short) if and only if $Y=\PP^1$ and $\fan_P$ is a translation of $\tail \fan$ by a lattice element for all but at most two points $P \in \PP^1$.
\end{cor}
\begin{proof}
  From the theorem we know that $\pdv(\fan)$ is isomorphic to
  $\pdv(\fan')$, where $\fan'$ has only trivial slices except from $\fan_0$ and $\fan_{\infty}$, but $\pdv(\fan')$ is toric as discussed in remark~\ref{sec:rem-toric}.
\end{proof}

\begin{exmp}
\label{sec:exmp-quadric}
  As an example for a divisorial fan we consider the set $\fan=\{\D_1,\D_2,\D_3,\D_4\}$. Each having locus $\PP^1$. The nontrivial slices of $\fan$ are shown in figure~\ref{fig:quadric}. Since $\dim N=2$ and $\dim Y =1$ the fan gives rise to a complete threefold with $T^2$ action.

 From remark~\ref{sec:rem-toric} we get that every $\pdv(\D_i)$ is isomorphic to toric variety of the cone spanned by 
$ \left(\begin{smallmatrix}
   1\\1\\1
 \end{smallmatrix}\right),
\left(\begin{smallmatrix}
   1\\0\\1
 \end{smallmatrix}\right),
\left(\begin{smallmatrix}
   -1\\-1\\-2
 \end{smallmatrix}\right)
$, which is $\mathbf{A}^3$. In particular $X$ is non-singular. By corollary~\ref{sec:cor-toric} we know that $\pdv(\fan)$ it is not toric, i.e. the toric charts are glued in a non-toric way. 

Note that all polyhedra in the slices have to be marked, since they originate from polyhedral divisors with complete locus.
  \begin{figure}[htbp]
    \subfigure[$\fan_0$]{
      \includegraphics[trim=0 10 0 0, width=3.6cm]{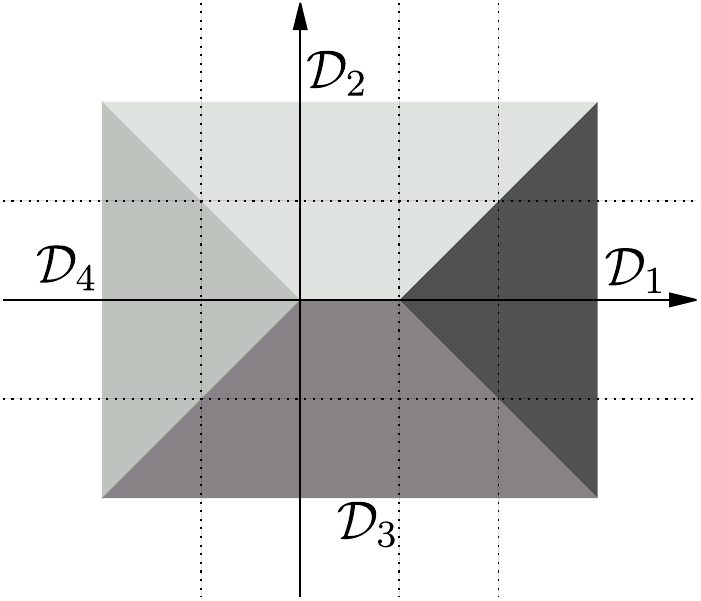}%
      \label{fig:quadric-first}
    }
    \subfigure[$\fan_1$]{
      \includegraphics[trim=0 0 0 20,height=2.9cm]{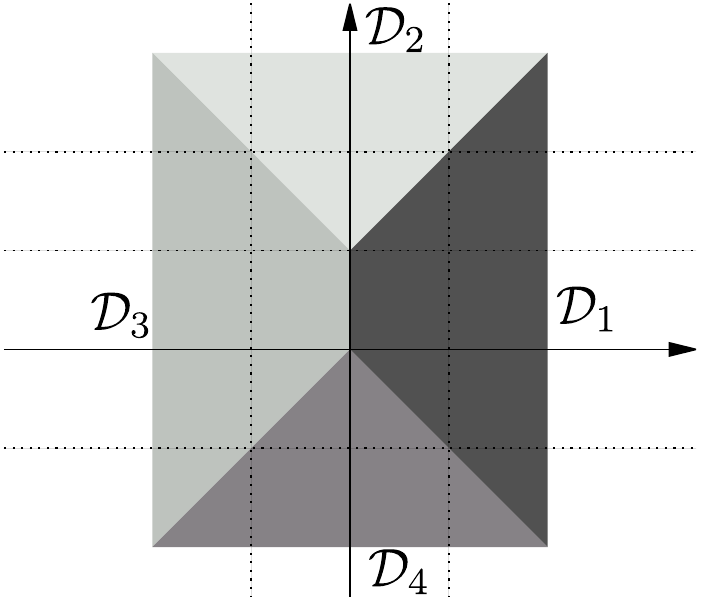}%
    }
    \subfigure[$\fan_\infty$]{
      \includegraphics[trim=0 20 0 0,height=2.7cm]{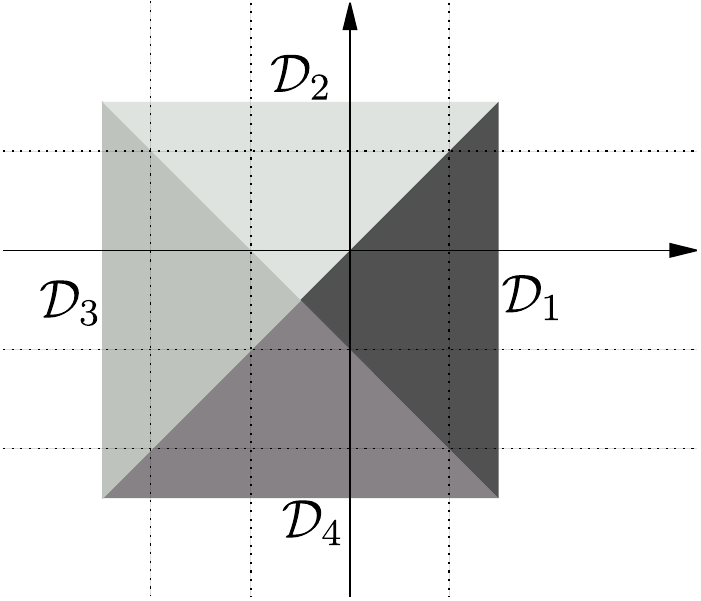}%
      \label{fig:quadric-last}
    }
    \caption{A divisorial fan}
    \label{fig:quadric}
  \end{figure}
\end{exmp}

\section{Invariant divisors}
\label{sec:invariant-divisors}
In this section we recall the description of Cartier and Weil divisors on T-varieties given in \cite{petersen-suess08}.

\label{sec:construction-cartier-div}
Let $\Sigma \subset N_\QQ$ be a complete polyhedral subdivision of $N$ consisting of tailed polyhedra. We consider continuous functions
$h:|\Sigma| \rightarrow \QQ$  which are affine on every polyhedron in $\Sigma$. Let $\Delta \in \Sigma$ a polyhedron with tail 
cone $\delta$, then $h$ induces a linear function $h^\Delta_0$ on $\delta = \tail \Delta$, by defining $h^\Delta_0(v):= h(P+v)-h(P)$ for some $P \in \Delta$. We call $h^\Delta_0$ the linear part of $h|_\Delta$. In this way we get a piecewise linear function $h_0$ on the tail fan $\tail(\Sigma)$.

\begin{defn}
  A piecewise affine function on  $\Sigma$ as above is called an {\em $\QQ$-support function}. It is called (integral) support function if it has integer slope and integer translation. The groups of support functions on $\Sigma$ are denoted by $\SF(\Sigma)$ and
 $\SF_\QQ(\Sigma)$ respectively.
\end{defn}

\begin{defn}
  Let $\fan$ be a divisorial fan on a curve $Y$. For every $P \in Y$ we get a polyhedral subdivision $\fan_P$ consisting of polyhedral coefficients.  We consider $\SF(\fan)$, the groups of ($\QQ$-){\em support functions} on $\fan$, i.e. collections $(h_P)_{P \in Y}\in \prod_{P \in Y}\SF(\fan_P)$ or $\prod_{P \in Y}\SF_\QQ(\fan_P)$ respectively with 
  \begin{enumerate}
  \item all $h_P$ have the same linear part $h_0$.
  \item only for finitely many $P \in Y$ $h_P$ differs from $h_0$.
  \end{enumerate}
\end{defn}

\begin{rem}
    We may restrict an element $h \in \SF(\fan)$ to a sub-fan or even to a polyhedral divisor $\D \in \fan$. The restriction will be denoted by $h|_\D$.
\end{rem}

\begin{defn}
  A support function $h \in \SF(\fan)$ is called principal, if $h =  u + D$ with $u \in M$ and $D=\sum_{P \in Y} a_i P$ is a principal divisor on $Y$. Here $h=u + D$ denotes the support function given by $h_P(v) = \langle u, v \rangle + a_P$.
\end{defn}

\begin{defn}
  A support function is called $T$-Cartier divisor if for every $\D \in \fan$ with complete locus $h|_\D$ is principal.  
  The group of $T$-Cartier divisors is denoted by $\tcadiv(\fan)$.

  The $\QQ$-support functions having a multiple being principal on all $\D$ with complete locus are denoted by $\tcadiv_\QQ(\fan)$ 
\end{defn}

The invariant rational functions on $X=\pdv(\fan)$ are given by $$K(X)^T=\bigoplus_{u \in M} K(Y) \cdot \chi^u.$$ Therefore principal support functions $h=\langle u,\cdot\rangle + \divisor(f)$ encode principal divisors $\divisor(f \cdot \chi^u)$. Given a $T$-Cartier divisor, after refining $\fan$ we can always assume that $h|_\D$ is principal on $\D \in \fan$.  In this way $\tcadiv(\fan)$ corresponds to the group of invariant Cartier divisors on $\pdv(\fan)$. The special condition for $\D$ with complete locus is needed, because they can not be refined and hence every invariant Cartier divisor has to be principal on $\pdv(\D)$ (c.f. \cite[Prop. 3.1.]{petersen-suess08})

Now we are going to describe Weil divisors on $\pdv(\fan)$. In general there are two types of $T$-invariant prime divisors. 
\begin{enumerate}
\item orbit closures of dimension $\dim T=\dim X - 1$. These correspond to 
   pairs $(P,v)$ with $P$ being point on $Y$ and $V$ a vertex of $\fan_P$,\label{item:dimT}
\item families of orbit closures of dimension $\dim T -1$. These correspond to 
  rays $\rho$ of $\sigma$ with $\deg \D \cap \rho = \emptyset$.\label{item:dimT-1}
\end{enumerate}
A ray of $\sigma$ with $\deg \D \cap \rho = \emptyset$ is called an {\em extremal ray}. The set of extremal rays is denoted by $\xrays(\D)$ or $\xrays(\fan)$ respectively if $\D \in \fan$.

\begin{prop}
  \label{sec:prop-cartier2weil}
  Let $h = \sum_P h_P$ be a Cartier divisor on $\D$. The corresponding Weil divisor is then given by 
  $$-\sum_{\rho} h_0 (n_\rho) D_\rho - \sum_{(P,v)} \mu(v) h_P(v) D_{(P,v)},$$
  where $\mu(v)$ is the smallest integer $k \geq 1$ such that $k \cdot v$ is a lattice point, this lattice point is a multiple of the primitive lattice vector: $\mu(v)v=\varepsilon(v) n_v$.
\end{prop}

\begin{thm}
  \label{sec:thm-canonical-divisor}
  For the canonical class of $X=\pdv(\fan)$ we have
  $$K_X = \sum_{(P,v)} \left((K_Y(P)+1) \mu(v) - 1 \right)\cdot D_{(P,v)} - \sum_{\rho}D_\rho$$
\end{thm}

Now we describe the global sections of $D_h$. Given a support function $h = (h_p)_P$ with linear part $h_0$. We define its associated polytope $$\Box_h := \Box_{h_0}:=\{u \in M \mid  \langle u, v \rangle  \geq  h_0(v) \; \forall_{v \in N}\},$$
and associate a dual function $h^*_P:\Box_h \rightarrow \wdiv_\QQ Y$ via
$$h^*(u):=\sum_P h_P^*(u)P := \sum_P \min_{\text{vertices}}(u-h_P)P.$$ 

\begin{rem} Let $h$ be a concave support function. Every affine piece of $h_P$ corresponds to a pair $(u,-a_u) \subset M \times \ZZ$. $h^*_P$ is defined to be the coarsest concave piecewise affine function with $h^*_P(u)=a_u$.
  
  We can reformulate this in terms of the tropical semi-ring with operation $\oplus=\min,\odot=+$. We may think of $h_P$ as given by tropical polynomials \tropical{\sum_{w \in I} (-a_w) * x^w}, then $\Box_{h}=\conv(I)$ and $h_P^*(w)=a_w$, i.e. $\Graph_{h^*_P}$ is the reflected lower newton boundary of the tropical polynomial for $h_P$.
\end{rem}

\begin{prop}
  \label{prop:global_sections}
  Let $h \in \tcadiv(\fan)$  define a Cartier divisor $D_h$ then
  $$\Gamma(X,\CO(D_h))=\bigoplus_{u \in \Box_h \cap M} \Gamma(Y,h^*(u))$$
\end{prop}

\begin{defn}
  For a cone $\sigma \in (\tail \fan)^{(n)}$ of maximal dimension in the tail fan and $P \in Y$ we get exactly one polyhedron $\Delta^\sigma_P \in \fan_P$ having tail $\sigma$.
  
  For a given concave support function $h = \sum h_P P$ 
  We have 
  $$h_P|\Delta^\sigma_P = \langle \cdot, u^h(\sigma) \rangle + a^h_P(\sigma).$$
  The constant part gives rise to a divisor on $Y$:
  $$h|_{\sigma}(0) := \sum_P a^h_P(\sigma) P.$$
\end{defn}

\begin{thm}
\label{sec:thm-ample}
  A $T$-Cartier divisor $D_h$ is semi-ample, iff all $h_P$ are concave and $-h|_\sigma(0)$ is semi-ample for all unmarked tail cones $\sigma$, i.e. $\deg -h|_\sigma(0) = -\sum_P a^h_P (\sigma) > 0$ or  a 
multiple of $-h|_\sigma(0)$ is principal.
\end{thm}

\begin{cor}
  A $T$-Cartier divisor $D_h$ is ample, iff all $h_P$ are strictly concave and $-h|_\sigma(0)$ is ample for all unmarked tail cones $\sigma$, i.e. $\deg -h|_\sigma(0) = -\sum_P a^h_P (\sigma) > 0$.
\end{cor}

\begin{cor}
  A T-Cartier divisor $D_h$ is nef iff $h$ is concave and $\deg h^*(u) \geq 0$ (i.e. $h^*(u)$ is nef) for every vertex $u$ of $\Box_h$.
\end{cor}
\begin{proof}
  After pulling back $D_h$ via an equivariant, birational and proper morphism we may assume that $D_h$ lives on a projective T-variety, i.e. there exists 
an ample $D_{h'}$. Now one checks easily that $D_h + \varepsilon D_{h'}= D_{h+\varepsilon h'}$ is ample for all $\varepsilon>0$ if and only if $h$ fulfils our preconditions.
\end{proof}

Using proposition~\ref{prop:global_sections} to determine $\dim \Gamma(X,D_h)$ we are also able to calculate intersection numbers. The following result was already obtained by Timashev  even in higher generality \cite{1034.14004}.

\begin{defn}
  For a function $h^*:\Box\rightarrow \wdiv_\QQ Y$ we define its {\em volume} to be
  $$\vol h^* := \sum_P \int_{\Box} h^*_P \vol_{M_\RR}$$

 For $h^*_1,\ldots, h^*_k$ we define their {\em mixed volume} by
  $$V(h^*_1,\ldots, h^*_k):= \sum_{i=1}^k (-1)^{i-1} \sum_{1\leq j_1 \leq \ldots j_i \leq  k} \vol(h^*_{j_1} + \cdots + h^*_{j_i})$$
\end{defn}

\begin{prop}\label{prop:intersection-numbers}
  Let $\fan$ be a divisorial fan on a curve with slices in $N \cong \ZZ^n$.
  \begin{enumerate}
  \item  If $D_h$ is nef, for the self-intersection number we get \label{item:prop-vol-vol}
    $$(D_h)^{(n+1)} = (n+1)!\vol h^*.$$
  \item  Let $h_1, \ldots, h_{n+1}$ define semi-ample divisors $D_i$ on $X(\fan)$. Then $$(D_1 \cdots D_{n+1}) = (n+1)!V(h^*_1, \ldots, h^*_{n+1}).$$\label{item:prop-vol-mixed}
  \end{enumerate}
\end{prop}
\begin{proof}
 If we apply (\ref{item:prop-vol-vol}) to every sum of divisors from $D_1,\ldots,D_{m+1}$ we get (\ref{item:prop-vol-mixed}) by the multi-linearity and symmetry of intersection numbers.
  
  To prove (\ref{item:prop-vol-vol}) we first recall that
$(D_h)^{m+1} = \lim_{\nu \rightarrow \infty} \frac{(m+1)!}{\nu^{m+1}} \chi(X,\CO(\nu D_h)),$
but for projective $X:=\pdv(\fan)$ and nef divisors the ranks of higher cohomology groups are asymptotically irrelevant \cite[Thm. 6.7.]{MR1919457} so we get
  $$(D_h)^{m+1} = \lim_{\nu \rightarrow \infty} \frac{(m+1)!}{\nu^{m+1}} h^0(X,\CO(\nu D_h)).$$

Note that $(\nu h)^*(u)=\nu \cdot h^*(\frac{1}{\nu}u)$. Now we can bound $h^0$ by
\begin{equation}
  \label{eq:h-bound}
\sum_{u \in \nu \Box_h \cap M} \left( \deg \lfloor \nu h^*\left({\textstyle \frac{1}{\nu}} u \right)\rfloor - g(Y) +1\right) \leq h^0(\CO(\nu D_h)) \leq \sum_{u \in \nu \Box_h \cap M} \deg \lfloor \nu h^*\left({\textstyle \frac{1}{\nu}}  u\right) \rfloor + 1.
\end{equation}

We have
\begin{eqnarray*}
   \lim_{\nu \rightarrow \infty} \frac{(m+1)!}{\nu^{m+1}} \sum_{u \in \nu \Box_h \cap M} \deg \lfloor \nu h^*\left({\textstyle \frac{1}{\nu}} u\right) \rfloor
&=& \lim_{\nu \rightarrow \infty} \frac{(m+1)!}{\nu^{m}} \sum_{u \in \Box_h \cap \frac{1}{\nu}M} \frac{1}{\nu} \deg \lfloor \nu h^*(u)\rfloor \\
&=& (m+1)! \int_{\Box_h} h^* \vol_{M_\RR},
\end{eqnarray*}
and  
$$\lim_{\nu\rightarrow \infty} \frac{1}{\nu^{m+1}} \sum_{u \in \nu \Box_h \cap M} (g - 1) = (g-1)\lim_{\nu\rightarrow \infty} \frac{\# (\nu\cdot \Box_h \cap M)}{\nu^{m+1}}=0$$
holds.

So if we pass to the limit in (\ref{eq:h-bound}) the term in the middle has to converge to $\vol h^*$.
\end{proof}

\begin{cor}
  A divisor $D_h$ is big if and only if $\vol h^* > 0$.
\end{cor}

\begin{exmp}
  We reconsider our divisorial fan of example~\ref{sec:exmp-quadric}. 
  By theorem~\ref{sec:thm-canonical-divisor} $3D_{(\infty,(\frac{1}{2},\frac{1}{2}))}$ is an anti-canonical divisor. By theorem~\ref{sec:prop-cartier2weil} it corresponds to the support function $h$ given by the tropical polynomials

  \begin{align*}
    h_0=&\mtrop{(-3)*x^{(3,0)}+0*x^{(0,3)}+0*x^{(-3,0)}+0*x^{(0,-3)}}\\
    h_1=&\mtrop{0*x^{(3,0)}+(-3)*x^{(0,3)}+0*x^{(-3,0)}+0*x^{(0,-3)}}\\
    h_\infty=&\mtrop{3*x^{(3,0)}+3*x^{(0,3)}+0*x^{(-3,0)}+0*x^{(0,-3)}}
  \end{align*}

  The functions $h$ and $h^*$ are sketched in figure~\ref{fig:canonical-divisor}.
    \begin{figure}[htbp]
      \subfigure[$h$]{
        \includegraphics[trim=0 470 0 0,clip=true, width=10cm]{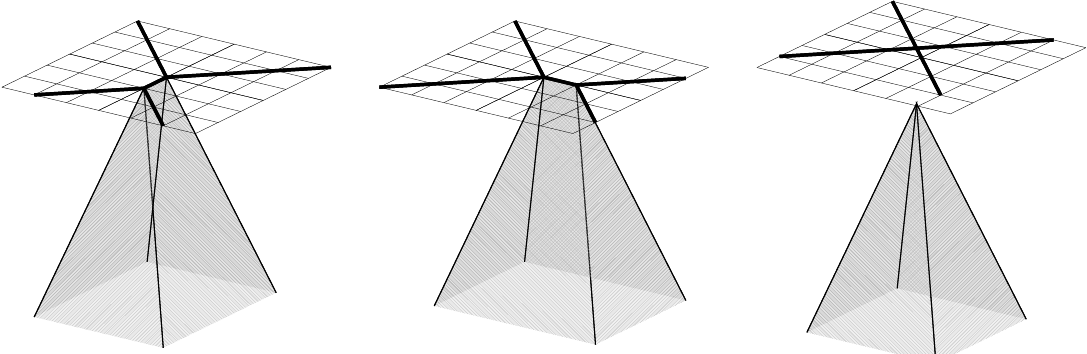}
      }
      \subfigure[$h^*$]{
        \includegraphics[trim=0 470 0 0,clip=true, width=10cm]{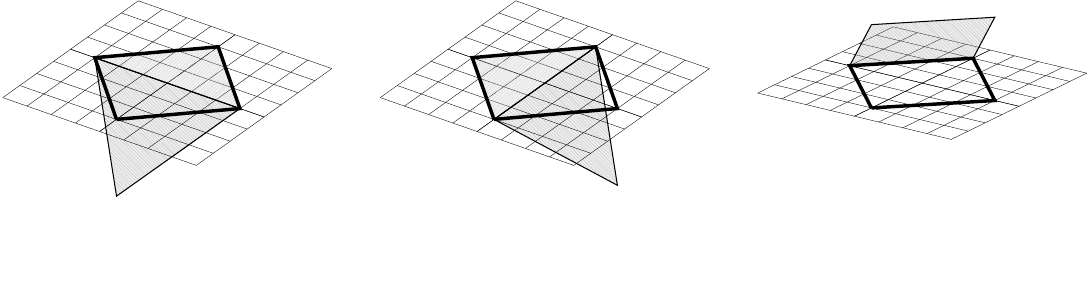}%
      }
      \caption{A canonical divisor on $Q^3$}
      \label{fig:canonical-divisor}
    \end{figure}

    Using the theorem~\ref{sec:thm-ample} we see that the anti-canonical divisor is ample.  From $h$ we see that $X$ has Fano index $3$. By theorem~\ref{prop:intersection-numbers} we get $(-K)^3=28$.

    So from the classification of Fano threefolds (\cite{0382.14013}) we know that $X$ is a quadric in $\PP^4$.
  \end{exmp}

\section{Singularities}
\label{sec:singularities}
First we give criteria to decide whether a polyhedral divisor on a curve describes a smooth variety. In principle this is only a reformulation of the well known fact, that affine varieties with good $\CC^*$ action are smooth if and only if they are open subsets of $\A^n$.

\begin{prop}
\label{sec:prop-D-smooth}
  A polyhedral divisor $\D$ having a complete curve $Y$ as its locus defines a smooth variety if and only if $Y=\PP^1$,
  $\D \sim \Delta_P \otimes P + \Delta_Q \otimes Q$ 
  and
  $$\delta:=\overline{\QQ^+ \cdot (\{1\} \times \Delta_P \cup \{-1\} \times \Delta_Q)} \subset \QQ\times N$$ 
  is a regular cone.
\end{prop}

\begin{proof}
  One direction is obvious, because the polyhedral divisor $\D$ describes the toric variety $X_\delta$ with a $T_N$-action induced by
  $N \hookrightarrow \QQ\times N$ \cite[section~11]{MR2207875}. For the other direction note that for $v \in \relint \sigma \cap N$ we get a positive grading 
  $$\Gamma(X,\O_X) = \bigoplus_{i \geq 0} \bigoplus_{\langle u,v \rangle = i} \Gamma(Y, \O(\D(u))).$$ 
  There is a full sub-lattice $M' \subset \sigma^\perp$ such that $\D(u)$ is integral for $u \in M'$. Because $\D$ is proper we get
  $\D(u)$ and $\D(-u)$ to be of degree $0$ and even principal $\D(u)=\divisor(f_u)$ and $\D(-u)=-\divisor(f_u)$ for $u \in M'$.
  This implies that $\Gamma(X,\O_X)_0 \subset K(X)$ is generated by $\{f^{\pm 1}_{u_1} \chi^{\pm u_1}, \ldots, f^{\pm 1}_{u_r} \chi^{\pm u_r}\}$ where the $u_i$ are elements of a $M'$-basis. We get $\Gamma(X,\O_X)_0 \cong k[M']$. 

  The following lemma implies that 
  $$X = \spec k[M'][\ZZ_{\geq 0}^{\dim \sigma}]$$
  holds and $\D$ has the claimed form.
\end{proof}

\begin{lem}
  Let $A_0=k[\ZZ^r]$ and $A = \bigoplus_{i\geq 0} A_i$ a finitely generated positive graded $A_0$-algebra. $A$ is a regular ring if
  and only if $A$ is a free generated as an $A_0$-algebra. 
\end{lem}
\begin{proof}
  We may choose a minimal homogeneous generating system $g_1, \ldots, g_l$ for $A$ and consider a maximal ideal $\m_0$ of $k[\ZZ^r]$ then
  $\m := A\m_0 + (g_1,\ldots, g_l)$ is a maximal ideal of $A$. If $A$ is regular and $l > \dim A - r$ then 
$\dim_k \m/\m^2 = \dim A < r + l$ and w.l.o.g. we may assume that there is a $g_i$ such that there is a (finite) homogeneous relation 
  $g_i = \sum_{|\alpha|\geq 2} b_\alpha g^\alpha$, with $b_\alpha \in A_0$, $\alpha \in \NN^l$. For degree reasons we may assume that $b_\alpha = 0$ 
  if $\alpha_i \neq 0$. But this contradicts the minimality of $\{g_1,\ldots, g_l\}$. So we have $l = \dim A - r$, thus $A$ is free as an $A_0$-algebra.
\end{proof}

\begin{thm}
\label{sec:thm-toric-sing}
  A polyhedral divisor $\D = \sum_P \Delta_P$ having an affine curve $Y$ as its locus, defines a smooth/$\QQ$-factorial/Gorenstein variety if and only if for every $P \in Y$ the cone
$$\delta_P := \overline{\QQ^+ \cdot (\{1\} \times \Delta_P)} \subset \QQ\times N$$
 is regular/simplicial/Gorenstein.

Moreover the singularity at $\pi^{-1}(P)$ is analytically isomorphic to the toric singularity $X_{\delta_P}$.
\end{thm}

\begin{proof}
  We may use the same argument as in \cite{MR2020670} for the case of an affine $\CC^*$-surface, because all properties can equivalently checked on the formal completion.

We have $A=\bigoplus_{u\in \sigmav \cap M} \Gamma(Y, \CO(\D(u)))$ and $Y = \spec A_0$. For a point $P \in Y$ we consider the maximal ideal $\m_P \subset A_0$ which gives rise to the maximal ideal $\mathfrak{M}_P = A\m^0_P + \bigoplus_{u \notin \sigma^\perp} A_u \subset A$.

Because $Y$ is smooth the completion of $A$ at $\m_P$ equals the completion of $\pdv(\D')$ at $\mathfrak{M}_0$, with $\D' = \Delta_P \otimes \{0\}$ on $Y'=\mathbf{A}^1$.

But now $\bigoplus_{u\in \sigmav \cap M} \Gamma(\mathbf{A}^1, \CO(\D')) \cong k[\deltav_P \cap \ZZ \times M]$ holds and our claim follows.

\end{proof}

\begin{prop}
  A polyhedral divisor $\D = \sum_P \Delta_P \otimes P$ having a complete curve $Y \not \cong \PP^1$ as its locus is not  $\QQ$-factorial. 
\end{prop}

\begin{proof}
 Because $g(Y)>0$ the free part of the Jacobian $J(Y)/J(Y)_{\text{tor}}$ is not finitely generated, thus we may find a point
 $P \in Y$ such that no multiple of $P$ is an element of $\langle Q \mid Q \in \supp \D \rangle \subset J(Y)$. Because $P \not \in \supp \D$ 
 the point $0 \in N$ is a vertex of $\Delta_P$ and $(P,0)$ represents an prime divisor $D_{(P,0)}$ of $X:=\pdv(\D)$. We claim that it is not $\QQ$-Cartier.

 In our case it enough to show that no multiple of $D_{(P,0)}$ is principal, because every invariant Cartier divisor on $X$ is even principal.
 Consider any $f \chi^u  \in K(X)^T$ such that $\ord_{D_{(P,0)}} (f \chi^u) = \ell > 0$.
 We have $\ord_{D_{(P,0)}} (f \chi^u) = \mu(0) (\ord_P(f) + \langle 0,u \rangle) = \ord_P(f)$. By the choice of $P$ there is a 
 $Q \in \supp \divisor(f) \setminus \supp \D$ with $Q\neq P$, this implies $\ord_{D_{(Q,0)}}(f \chi^u) = \ord_Q(f) \neq 0$ and thus $\divisor (f \chi^u) \neq \ell \cdot D_{(P,0)}$
\end{proof}

\begin{rem}
  In the case $Y \cong \PP^1$ we may decide wether $\D = \Delta_1 \otimes P_1 + \ldots + \Delta_s \otimes P_s$ defines a ($\QQ$-)factorial variety by means of linear algebra.
  Consider the following matrix:
$$
A(\supp \D) =\begin{pmatrix}
  -1 & -1 & \ldots &  -1& 0 \\
  \hline
  \mu(v_{P_1}^1) & 0 & \ldots & 0 & \mu(v_{P_1}^1) v_{P_1}^1 \\
  \vdots & \vdots  & & \vdots &\vdots \\
  \mu(v_{P_1}^{r_1})& 0 & \ldots & 0 & \mu(v_{P_1}^{r_1}) v_{P_1}^{r_1} \\
         &         & \ddots & & \\
  0 & 0 & \ldots &  \mu(v_{P_s}^1)    &  \mu(v_{P_s}^1) v_{P_s}^{1} \\
  \vdots & \vdots  & & \vdots &\vdots \\
  0 & 0 & \ldots &  \mu(v_{P_s}^{r_s}) & \mu(v_{P_s}^{r_s}) v_{P_s}^{r_s} \\
  \hline
  0 & 0 & \ldots & 0 & \rho^1 \\
  \vdots & \vdots  &  & \vdots &\vdots \\
  0 & 0 & \ldots & 0 & \rho^{r_0}   
\end{pmatrix}
$$
  We have $r$ rows, one ensuring that $\sum a_i = 0$, one for every vertex $v_{P_j}^i$ of a nontrivial coefficient $\Delta_{P_j}$ and one for every $\rho \in \xrays(\D)$. There are $s+n$ columns, one for every $P_i \in \supp \D$ and $n=\dim N$ for the coordinates of the vertices $v_{P_j}^i$.

  $\pdv(\D)$ is factorial iff the columns of the matrix generate $\ZZ^r$. It is $\QQ$-factorial iff they generate $\QQ^r$ or equivalently iff the rows are linearly independent.
\end{rem}
\begin{proof}
\textit{Step 1}
  We may choose a {\em generic} point $Q \notin \supp \D$ and we get a corresponding prime divisor $(Q,0)$. Now every Weil divisor of $X$ is linearly equivalent to one of the form 
$$\sum_{ij} a_{ij} \cdot D_{(P_i, v_{P_i}^j)} + \sum_{i} b_{i} \cdot D_{\rho^i} + c \cdot D_{(Q,0)},$$
because $D=\sum_{i=1}^r c_i \cdot D_{(Q_i,0)}$ is principal for $\sum_{i=1}^k c_i = 0$ and $Q_i \notin \supp \D$, since
$D=\divisor(f \chi^0)$, where $\divisor(f) = \sum c_i Q_i$.

\textit{Step 2}
  Every principal divisor $D$ of $X$ is given by an principal divisor of $\PP^1$, i.e. a Weil divisor $D_Y = \sum_P a_P P$ of degree $0$ and a weight $u \in M$. Because of \textit{step 1} it is sufficient  to consider those principal divisors such that $\supp D_Y \subset \supp \D \cup \{Q\}$.
  
  The coefficients for the  prime divisors $D_{(Q,0)}$, $D_{(P_i, v_{P_i}^j)}$ and $\rho^i$ of the corresponding Weil divisor on $X$ can be obtained as rows of the vector $A \cdot (a_{P_1},\ldots ,a_{P_s},u)^t$. So $X$ is factorial if this mapping is surjective and $\QQ$-factorial if it is $\QQ$-continuation is surjective.
\end{proof}

\begin{cor}
  Let $\D$ be a polyhedral divisor with $\loc \D=\PP^1$ then $\X(\D)$ is \mbox{$\QQ$-factorial} iff 
    $$\# \xrays(\D) + \sum_{P \in Y} (\# \Delta_P^{(0)} - 1) < \dim(N).$$

\end{cor}
\begin{proof}
  By the last remark it is enough to see that $A$ always has maximal rank.
\end{proof}

If $\pdv(\fan)$ is $\QQ$-factorial we can compute the Picard rank of a T-variety:
\begin{thm}[\cite{petersen-suess08}, theorem 3.20.]
\label{sec:thm-picard-rank}
  Assume $X=\X(\fan)$ to be a complete \mbox{$\QQ$-factorial} variety. Then its Picard number is given by
  $$\rho_X= 1 + \# \xrays(\fan) + \sum_{P \in Y} (\# \fan_P^{(0)} - 1) - \dim(N).$$
\end{thm}

\begin{rem}
\label{sec:rem-gorenstein-index}
  We may also check the ($\QQ$-)Gorenstein property in the case of complete $Y$ by considering 
  $A(\supp \D \cup \supp K_Y)$ as above. Because it always has maximal rank there is at most
  one solution $\left(\begin{smallmatrix}
 a  \\
 u_K
\end{smallmatrix}\right) \in  \QQ^{\supp \D \cup \supp K_Y} \oplus M_\QQ$ of the following equation for the coefficients of the canonical divisor due to the theorems~\ref{sec:thm-canonical-divisor} and \ref{sec:prop-cartier2weil}.

$$A \cdot
\left(\begin{smallmatrix}
 a_1 \\
 \vdots \\
 a_s\\
 u_K
\end{smallmatrix}\right)
=(K_X):=\left(\begin{smallmatrix}
0\\
(K_Y(P_1)+1)\mu(v^1_{P_1})-1\\
(K_Y(P_1)+1)\mu(v^2_{P_1})-1\vspace{-2mm}\\
\vdots\\
(K_Y(P_s)+1)\mu(v^{r_s}_{P_s})-1\\
-1\vspace{-2mm}\\
\vdots\\
-1
\end{smallmatrix}\right)
$$

Now $\D$ is $\QQ$-Gorenstein of index $\ell$, iff a solution exists and $\ell$ is the smallest integer such that
$\ell \cdot \sum a_i P_i$ is principal.
\end{rem}

\begin{prop}
  Let $\D$ be a prolyhedral divisor with a complete curve $Y$ as its locus, such that $X=X(\D)$ is $\QQ$-Gorenstein. Then 
  $X(\D)$ is log-terminal if and only if $Y=\PP^1$ and for every ray $\rho \notin \xrays(\D)$ we have $\sum_{P \in Y} (1-\frac{1}{\mu(\rho_P)}) < 2$, where $\rho_P$ is the unique vertex $v \in \Delta_P$ such that $v + \rho$ is a face.
\end{prop}
\begin{proof}
  We may resolve the singularities by first replacing $\D$ by a divisorial fan $\fan$ consisting of $\D_1 = \D + \emptyset \otimes Q_1$ and $\D_2 = \D + \emptyset \otimes Q_2$, for $Q_1 \neq Q_2 \in Y$. 
  $\D_1, \D_2 \subset \D$ define two birational maps which glue to a birational proper map $\psi: \tilde{X}=\pdv(\fan) \rightarrow \pdv(\D)$, having $\sum_{\rho} D_\rho$ as exceptional divisor, where $\rho$ runs over all non-extremal rays of $\tail \D$.

$\D_1$ and $\D_2$ may still have toric singularities which correspond to the cones $\delta_P$. These can be resolved by 
toric methods (Cf. \cite[2.6]{MR1234037}) which result in a subdivision of polyhedral coefficients. 
But toric pairs (X,B) are log-terminal so it is enough to control the discrepancies of $\psi$.

Because $X$ is $\QQ$-Gorenstein we have $\ell K_X = \divisor (f \chi^{\ell u_K})$. So the pull back via $\psi$ is given by the same function (by identifying the function fields of $X$ and $\tilde{X}$). By proposition~\ref{sec:prop-cartier2weil} and theorem~\ref{sec:thm-canonical-divisor} we get
$$K_{\tilde{X}}-\psi^*K_X=\sum_\rho (1-\langle  u_K, \rho \rangle)D_\rho,$$
where $\rho$ runs over the non-extremal rays of $\tail \fan$.

  We consider $A:=A(\supp \D \cup \supp K_Y)$ as above. We choose a vertex $v \in \deg \D \subset \tail \D$. Then for every polyhedral coefficient we have a vertex $v_P \in \D_P$, such that $v=\sum_P v_p$.
Now, for every ray $\rho \subset \tail \D$ we have the relation
$\langle u_K, \rho \rangle < 0$ for non-extremal rays this is exactly the log-terminal condition for extremal rays we even know that $\langle u_K, \rho \rangle -1$.

We take the corresponding row of the vertices $v_P$ of $A \left(\begin{smallmatrix}a \\ u_K \end{smallmatrix}\right)=(K_X)$ divide by $\mu(v_P)$  and add them to  the first row. 
  This gives 
  $$\left(0,\ldots,0, v \right){a \choose u_K } = \deg K_Y + \sum_P \left(1 - \frac{1}{\mu(v_P)}\right)$$
The left hand side has to be negative, since $v \in \deg \D$, but the right hand side is negative if and only if $Y=\PP^1$ and
$\sum_P \left(1 - \frac{1}{\mu(v_P)}\right)<2=-\deg K_{\PP^1}$.
\end{proof}

\begin{rem}
\label{sec:rem-quotient-sing}
  Consider the case of a polyhedral divisor $\D = \sum_P [a_P,\infty)\otimes P$
  with locus $\PP^1$ and $M=\ZZ$. The proposition tells us that $\pdv(\D)$ is log-terminal if either at most two $a_P$ are non-integers, which means that we have a cyclic quotient singularity or three values $a_P$ are non-integers, which we denote by $\frac{e_i}{m_i}, i=1,2,3$ and $(m_1,m_2,m_3)$ is one of the Platonic triples $(2,2,m)$, $(2,3,3)$, $(2,3,4)$, $(2,3,5)$. This means that we have a quotient singularity. This is a classical result of Brieskorn\cite[Satz 2.10.]{MR0222084}. He denoted the  corresponding quotient singularity (or its resolution graph, respectively) by \[\langle \sum_P \lceil a_P \rceil; m_1,-e_1;m_2,-e_2;m_3,-e_3 \rangle,\]
where the entries $-e_i$ are understood as values in $\ZZ/m_i\ZZ$. This data encodes the resolution graph of the singularity as follows (c.f. \cite{MR0222084}). Look at the continued fraction 
\[\frac{p_i}{q_i}=b_1-\frac{1}{b_2-\frac{1}{b_3-\dots}}.\]
The resolution graph consisting of a chain of vertices with self intersection numbers $-b_i$ is denoted by $\langle p_i,q_i \rangle$. Now, the graph which is obtained by gluing three of this chains together at a new vertex with self intersection number $-b$ is denoted by $\langle b; q_1,p_1;q_2,p_2;q_3,p_3 \rangle$.

\end{rem}

\section{Del Pezzo $\CC^*$-surfaces}
\label{sec:del-pezzo}
As an application we want to classify log del Pezzo surfaces, i.e. projective surfaces with $\QQ$-Cartier and ample canonical divisor.

This was done completely for Gorenstein index $\ell \leq 2$ in \cite{MR1009466}.  We will concentrate on the case of $\CC^*$ surfaces with $\rho \leq 2$ and discuss the situation for $\ell = 3$ in full detail.

Here we are interested in non-toric $\CC^*$ surfaces only. There are several reasons for this decision. One reason is that toric surfaces admit infinitely many different $\CC^*$-action that means we are not able to give a finite classification in our language.
One of the features of our description is, that we are able to describe families of del Pezzo surfaces. But toric ones do not live in families.
On the other hand the machinery to handle the toric cases can be found in \cite{Dais:arXiv0709.0999} for Picard number $1$ and $\ell = 3$ and in
\cite{CambridgeJournals:7229856} for the general toric case.

\subsection{Multidivisors}
For the case of $\CC^*$-surfaces we can simplify our description by divisorial fans. A $\CC^*$-surface is completely determined by the marked slices of a corresponding divisorial fan.
Here the slices are subdivisions of $\QQ$ which can be described by their sets of boundary points. This lead to the following 

\begin{defn}
  A {\em multidivisor} is formal finite sum 
  $$\fan =  \sum_P \fan_P \otimes P \text{ together with  marks }\subset \{\QQ^-,\QQ^+\},$$

  Here, $P$ runs over the points of $Y$ and $\emptyset \neq \fan_P \subset \QQ$ are finite sets of rational numbers. And only finitely many of them are allowed to differ from $\{0\}$.

  a mark on $\QQ^-$, $\QQ^+$ is only allowed if $\sum a^+_P > 0$ or $\sum a_P^- < 0$ respectively. Here $a_P^\pm$ denotes the least or greatest value in $\fan_P$, respectively.
\end{defn}

The condition on $\deg \fan$ ensures the the set of marked slices really can be realised by a divisorial fan. 

Multidivisors form a semi-group which includes the set of regular $\QQ$-divisors as a subgroup. So we have the relation of linear equivalence on them 
$$\fan\sim\fan' \Leftrightarrow \fan = \fan + \divisor(f).$$

Now theorem~\ref{sec:thm-fan-iso-codim1} implies the following
\begin{cor}
  \label{sec:cor-multidiv-isomorph}
  For every normal $\CC^*$-surface $X$ there is a multidivisor $\fan$, such that $\pdv(\fan)=X$ and two multidivisors $\fan,\fan'$ give equivariantly  isomorphic $\CC^*$-surfaces iff $\fan \sim \pm \varphi^*\fan'$ for an isomorphism $\varphi:Y\rightarrow Y'$.
\end{cor}

\begin{rem}
  A multidivisor $\fan$ describes a toric surface iff $Y = \PP^1$ and
  it has at most two coefficients which are not single integers, i.e.
  $$\fan \sim  \fan'=\fan'_Q \otimes Q + \fan'_P \otimes P.$$
\end{rem}

A description of affine $\CC^*$ surfaces and their completions by sets of $\QQ$-divisors (which is essentially the same as what we do) was given in \cite{MR2020670} and \cite{MR2327238}. A description of non-singular complete $\CC^*$-surfaces by graphs can be found in \cite{MR0460342}. 

In both papers one finds a classification of fixed points on $\CC^*$ surfaces:
\begin{description}
\item[parabolic] the fixed point is adjacent to exactly one maximal orbit. Parabolic fixed points gather in curves.
\item[elliptic] the fixed points is adjacent to all maximal orbits in their neighbourhood. Elliptic fixed points are always isolated.
\item[hyperbolic] the fixed points is adjacent to exactly two maximal orbits. Hyperbolic fixed points are always isolated.
\end{description}

\begin{rem}
\label{sec:rem-fix-points}  
From \cite{MR2020670} and \cite{MR2207875} we know that parabolic fixed points correspond to a unmarked side and a point on $Y$, while elliptic fixed points correspond to a marked side of our multidivisor and can be seen as the result of a contraction of a curve of parabolic fixed points \cite[theorem~10.1.]{MR2207875}.
\end{rem}

\begin{prop}
  Every non-toric log Del Pezzo $\CC^*$-surface has one or two elliptic fixed points.
\end{prop}

\begin{proof}
  Remark~\ref{sec:rem-fix-points} implies that we have at most $2$ elliptic fixed points on a $\CC^*$-surface.

  We now assume that we have a log del Pezzo $\CC^*$-surface without elliptic fixed points. It corresponds to a multidivisor on $\PP^1$ without marks and with at least tree coefficients which are not pure integers.

We denote by $h$ the support function corresponding to the canonical divisor of $\pdv(\fan)$. Because both rays in the tail fan are extremal by  theorem~\ref{sec:thm-canonical-divisor} we have 
$$h_0(v)=
\begin{cases}
  v & \text{, for }v > 0\\
  -v & \text{, else}
\end{cases}
$$
Let $\fan_P$ one of the non-trivial coefficients. From theorem~\ref{sec:thm-canonical-divisor} and proposition~\ref{sec:prop-cartier2weil} we know that
\begin{eqnarray*}
  \frac{1}{m^-}-1-K_Y(P)&=&h_P(\frac{e^-}{m^-})=-\frac{e^-}{m^-}+a^-\\
  \frac{1}{m^+}-1-K_Y(P)&=&h_P(\frac{e^+}{m^+})=\frac{e^+}{m^+}+a^+.
\end{eqnarray*}
This implies
$a^- \leq \lfloor \frac{e^-}{m^-} \rfloor - K_Y(P)$ and
$a^+ \leq -\lceil \frac{e^+}{m^+} \rceil- K_Y(P)$.
Since $\lfloor \frac{e^-}{m^-} \rfloor -  \lceil \frac{e^+}{m^+} \rceil \leq -1$ and there are at least $3$ non-trivial coefficients we get $\sum_P a_P^-\leq -2-\deg K_Y$ or  $\sum_P a_P^+ \leq -2 - \deg K_Y $. By theorem \ref{sec:thm-ample} this contradicts the ampleness of $D_{-h}$.
\end{proof}

\subsection{Classification}
We can use corollary~\ref{sec:thm-picard-rank} to obtain the structure of multidivisors giving rise to a $\CC^*$-surface with Picard rank $1$. There are two cases:
\begin{enumerate}
\item Marks on both sides and exactly one coefficient $\fan_P$ with $\#\fan^{(0)}_P=2$, i.e. two elliptic and one hyperbolic fixed point.\label{item:pic1-two-elliptic}
\item One mark and all coefficients have cardinality $1$, i.e. one elliptic fixed point and a curve consisting of parabolic fixed points.\label{item:pic1-one-parabolic}
\end{enumerate}

Now we want to classify non-toric log del Pezzo surfaces of Picard rank $1$. Non-toric surfaces correspond to multidivisors with at least three coefficients different from a single integer.

\subsubsection*{Classification in case \ref{item:pic1-two-elliptic}}
Because of remark~\ref{sec:rem-quotient-sing} we have to consider three or four coefficients, namely at most three non-integral rationals and one set of cardinality two.

For a given multidivisor sections~\ref{sec:invariant-divisors} and \ref{sec:singularities} give us the tools to decide whether it is log del Pezzo of a given index or not. So in order to enumerate all log Del Pezzo $\CC^*$-surfaces of index $\ell$ with two elliptic fixed points we consider the following multidivisor
$$\fan=\{a_1\}\otimes P_1 +  \{a_2\}\otimes P_2 + \{a_3\}\otimes P_3 + \{a_-,a_+\}\otimes P_0,$$
with $a_+>a_-$, $a_+ + \sum_{i=1}^3a_i > 0$, $a_- + \sum_{i=1}^3a_i < 0$ and $a_i=\frac{p_i}{q_i}$, for $i=+,-,1,2,3$ and $1\leq p_i < q_i$ for $i=1,2,3$.

We now try to derive bounds for $q_1,q_2,q_3,q_+,q_-$ in order to get a finite classification. 

The left hand sides of the subdivisions result from a polyhedral divisor
$$\D_- = (-\infty,a_1]\otimes P_1 + (-\infty, a_2]\otimes P_2 + (-\infty,a_3]\otimes P_3 + (-\infty,a_-]\otimes P_0$$
and the right hand sides from 
$$\D_+ = [a_1,\infty)\otimes P_1 + [a_2,\infty)\otimes P_2 + [a_3,\infty)\otimes P_3 + [a_+,\infty)\otimes P_0.$$

Because of remark~\ref{sec:rem-quotient-sing} we have to consider five cases
\begin{enumerate}
\item[1a)] 
$(q_1,q_2,q_3)$ may be one of the Platonic triples
 $(2,2,k),(2,3,3)$, $(2,3,4)$, $(2,3,5)$ and $q_+= q_- =1$.
 
 From $$(\frac{p_3}{q_3}+a_+ + \frac{1}{2}+\frac{1}{2})u_+=1-\frac{1}{q_3}-\frac{1}{2}-\frac{1}{2},$$
with $\ell u_+ \in \ZZ$, we get $(p_3+p_+q_3)|\ell$, and analogue from
  $$(\frac{p_3}{q_3}+a_-+\frac{1}{2}+\frac{1}{2})u_-=1-\frac{1}{q_3}-\frac{1}{2}-\frac{1}{2}$$
we get $(p_3+p_-q_3)|\ell$. Together with $a_++a_- \geq 1$ this implies $q_3 < \ell$. 

\item[1b)] W.l.o.g. $q_3=1$ and $(q_1,q_2,q_+)$ and $(q_1,q_2,q_-)$ are both one of the triples  $(2,3,3)$, $(2,3,4)$, $(2,3,5)$.\\
\item [1c)] $q_3=q_+=q_-=1$ and $q_1,q_2 \in \NN$.\\
 W.l.o.g we may assume that $a_++a_1+a_2>\frac{1}{2}$.
And we have $(a_++ a_2 + a_3)u_k= -\frac{1}{q_1} -\frac{1}{q_2}$. 
Because $u_K \in \frac{1}{\ell}\ZZ$ this implies
that $(q_1q_2p_+ + p_1q_2 +p_2q_1)$ divides $\ell (q_1+q_2)$, in particular $(\frac{q_1q_2}{2}) \leq (q_1q_2p_+ + p_1q_2 +p_2q_1) \leq \ell (q_1+q_2)$, which implies that one of $q_1,q_2$ has to be less than $4\ell$. W.l.o.g. let $q_1 < 4\ell$
On the other hand  for suitable $\lambda \neq \mu$
\begin{eqnarray*}
\lambda (q_1q_2p_+ + p_1q_2 +p_2q_1) =  \ell (q_1+q_2)\\
\text{ and }\mu (q_1q_2p_- + p_1q_2 +p_2q_1) = \ell (q_1+q_2)
\end{eqnarray*}
this implies that
\begin{eqnarray*}
(\lambda p_1-\ell)q_2 = (\ell - \lambda p_2+q_2p_+)q_1\\
\text{ and }(\mu p_1-\ell)q_2 = (\ell - \mu p_2+q_2p_-)q_1.
\end{eqnarray*}
From this we get that $q_2 | (\lambda-\mu)p_2q_1$.

Let's assume that $q_2 > 2\ell q_1$ then we must have $|\lambda|,|\mu| < \ell+1$. Hence, $|\lambda-\mu| < 2\ell +2$. Because $p_2$ and $q_2$ are coprime we get $\frac{q_2}{4\ell} < \frac{q_2}{q_1} < |\lambda-\mu| < 2 \ell + 2$. So we get the bounds $q_1 < 4\ell$ and $q_2 < 8(\ell^2 + \ell)$

\item [1d)] $q_3=1,q_2=2$ and furthermore $q_+,q_- \in \{1,2\}$ and $q_1 \in \NN$.

  In a similar way as in the last case we get $q_1<4 \ell - 4$
\item[1e)]$q_1=q_2=2$, $q_3=1$, $q_-,q_+ \in \NN$.
Because $a_++\frac{1}{2}+\frac{1}{2} > 0$ and $a_-+\frac{1}{2}+\frac{1}{2} < 0$ we get $p_++q_+>0$ and $p_-+q_- < 0$

From $$(\frac{p_+}{q_+}+\frac{1}{2}+\frac{1}{2})u_+=1-\frac{1}{q_+}-\frac{1}{2}-\frac{1}{2},$$
with $\ell u \in \ZZ$, we get $(p_++q_+)|\ell$, and analogue $(p_-+q_-)|\ell$.

Now we consider the affine function $h_{P_0}|{\scriptscriptstyle[\frac{p_-}{q_-},\frac{p_-}{q_-}]}$ realising the canonical divisor locally, i.e. $h(a_-)=1-\frac{1}{q_-}+K_Y(P_0)$ and $h(a_+)=1-\frac{1}{q_+}+K_Y(P_0)$.

We get $$h_{P_0}(1)=\frac{(p_++q_+)-(p_-+q_-)}{(p_++q_+)q_--(p_-+q_-)q_+} + 1 + K_Y(P_0).$$

This has to be an element of $\frac{1}{\ell}\ZZ$ so we have $\frac{1}{\ell} \leq \frac{(p_++q_+)-(p_-+q_-)}{(p_++q_+)q_--(p_-+q_-)q_+}$. W.l.o.g. let $q_-\leq q_+$ this implies, that $q_- \leq \ell$ and $q_+ \leq 2\ell^2$.


\end{enumerate}

\subsubsection*{Classification in case \ref{item:pic1-one-parabolic}}
For case~\ref{item:pic1-one-parabolic} we have exactly three non-integral rationals as coefficients 
$a_1 = \frac{p_1}{q_1}$, $a_2 = \frac{p_2}{q_2}$, $a_3 = \frac{p_3}{q_3}$.

Because of \ref{sec:cor-multidiv-isomorph}. 
we may assume that we have a mark on $\QQ^+$ and $0 < a_1,a_2 < 1$. So the right sides of all subdivisions result from one polyhedral divisor 
$$\D_+ = [a_1,\infty)\otimes P_1 + [a_2,\infty)\otimes P_2 + [a_3,\infty)\otimes P_3.$$

Now note that for $u_K=\frac{r}{s}$ we have
$$(a_1+ a_2 + a_3)\frac{r}{s} = 1-\frac{1}{q_1} -\frac{1}{q_2}-\frac{1}{q_3}$$
with $K_{\pdv(\D)}= \frac{1}{s}\divisor(f \chi^{r})$ C.f

As $u_K \in \frac{1}{\ell}\ZZ$ we must have $a_1+a_2+a_3 < 2\ell$.
We know that $(q_1,q_2,q_3)$ is one of the Platonic triples.
For  $(2,3,3)$, $(2,3,4)$, $(2,3,5)$ there are only finitely many possibilities for corresponding $p_1,p_2,p_3$. $(36 \ell = 2\ell \cdot (1\cdot 2 \cdot 2 + 1\cdot 2 \cdot 3 + 1\cdot 2 \cdot 4))$

For  $(q_1,q_2,q_3) =(2,2,m)$ we get
$u_K = q_3+p_3$ this implies that $|q_3+p_3| \leq \ell$. Because of \ref{sec:rem-quotient-sing} we know that $X$ has 
cyclic quotient singularity $Z_{q_3,p_3}$. This singularity has Gorenstein index $\frac{q_3}{\gcd(q_3,1-p_3)}$. This implies that $q_3 < \ell^2$. So we have finitely many pairs $(p_3,q_3)$ to check. 

\subsubsection{Enumerating the results}
In both cases for a given index $\ell$ there is a finite number of multidivisor to be checked if they have the desired index and if they fulfil the Fano property. We check the Fano property by using theorem~\ref{sec:thm-ample} and the Gorenstein index with the help of remark~\ref{sec:rem-gorenstein-index} or theorem~\ref{sec:thm-toric-sing}, respectively. 

Doing this for $\ell=1,2,3$ we are able to state the following theorems. The results for $\ell=1,2$ may be compared with \cite{MR2227002}, \cite{MR1933881}, \cite{MR961326}.

\begin{thm}
  There are the following Gorenstein log del Pezzo $\CC^*$-surfaces with Picard rank $1$.
  \begin{enumerate}
  \item  11 surfaces with two elliptic fixed points \label{item:idx1-2fp}
  \item  1 family over $\A^1$ of surfaces with two elliptic fixed points \label{item:idx1-family}
  \item  1 surface with one elliptic fixed point \label{item:idx1-1fp}
  \item  5 toric surfaces (c.f. \cite{MR2298756},\cite{Dais:arXiv0709.0999})
  \end{enumerate}
\end{thm}

\begin{spacing}{1.2}
  \begin{center}
    \begin{tabular}[htbp]{lcccc}
\toprule
   & multidivisor & marks & deg. & sing.\\
  \cmidrule{2-5}
  (\ref{item:idx1-2fp}) &$\{-2,0\} \otimes P_0\ +\ \{\frac{1}{2}\}\otimes P_1  \ +\  \{\frac{1}{2}\}\otimes P_2$  &  $\QQ^\pm$ & $2$  &    $2A_3A_1$\\
  &$\{-1,0\} \otimes P_0\ +\ \{\frac{1}{3}\}\otimes P_1  \ +\  \{\frac{1}{2}\}\otimes P_2$            &  & $5$  &    $A_4$\\
  &$\{-1,0\} \otimes P_0\ +\ \{\frac{1}{3}\}\otimes P_1  \ +\  \{\frac{1}{3}\}\otimes P_2$            &  & $2$   &   $A_5A_2$\\
  &$\{-1,0\} \otimes P_0\ +\ \{\frac{1}{4}\}\otimes P_1  \ +\  \{\frac{1}{2}\}\otimes P_2$            &  & $3$   &   $A_5A_1$\\
  &$\{\frac{-3}{2},-1\} \otimes P_0\ +\ \{\frac{2}{3}\}\otimes P_1 \  +\  \{\frac{1}{2}\}\otimes P_2$ &  & $4$   &   $D_5$\\
  &$\{\frac{-3}{2},-1\} \otimes P_0\ +\ \{\frac{2}{3}\}\otimes P_1 \  +\  \{\frac{2}{3}\}\otimes P_2$ &  & $1$   &   $E_6A_2$\\
  &$\{\frac{-3}{2},-1\} \otimes P_0\ +\ \{\frac{3}{4}\}\otimes P_1 \  +\  \{\frac{1}{2}\}\otimes P_2$ &  & $2$   &   $D_6A_1$\\
  &$\{\frac{-4}{3},-1\} \otimes P_0\ +\ \{\frac{2}{3}\}\otimes P_1 \  +\  \{\frac{1}{2}\}\otimes P_2$ &  & $3$   &   $E_6$\\
  &$\{\frac{-4}{3},-1\} \otimes P_0\ +\ \{\frac{3}{4}\}\otimes P_1 \  +\  \{\frac{1}{2}\}\otimes P_2$ &  & $1$   &   $E_7A_1$\\
  &$\{\frac{-5}{4},-1\} \otimes P_0\ +\ \{\frac{2}{3}\}\otimes P_1 \  +\  \{\frac{1}{2}\}\otimes P_2$ &  & $2$   &   $E_7$\\
  &$\{\frac{-6}{5},-1\} \otimes P_0\ +\ \{\frac{2}{3}\}\otimes P_1 \  +\  \{\frac{1}{2}\}\otimes P_2$ &  & $1$   &   $E_8$\\
  \cmidrule{2-5}
  (\ref{item:idx1-family}) & $\{-2,-1\} \otimes P_0+\{\frac{1}{2}\}\otimes P_1+\{\frac{1}{2}\}\otimes P_2 + \{\frac{1}{2}\}\otimes P_3$ & $\QQ^\pm$ & $1$ & $2D_4$ \\
  \cmidrule{2-5}
  (\ref{item:idx1-1fp}) & $\{\frac{-1}{2}\}\otimes P_1\ +\   \{\frac{1}{2}\}\otimes P_2\ +\   \{\frac{1}{2}\}\otimes P_3$ & $\QQ^+$ & $2$ & $D_43A_1$\\
\bottomrule
\end{tabular}
\end{center}
\end{spacing}

\begin{thm}
  There are $16$ log del Pezzo $\CC^*$-surfaces with Picard rank $1$ and index $\ell=2$.
  \begin{enumerate}
  \item  9 surfaces with two elliptic fixed points \label{item:idx2-2fp}
  \item  1 surfaces with one elliptic fixed point  \label{item:idx2-1fp}
  \item  7 toric surfaces (c.f. \cite{MR2298756},\cite{Dais:arXiv0709.0999})
  \end{enumerate}
\end{thm}

\begin{spacing}{1.2}
  \begin{center}
    \begin{tabular}[htbp]{lcccc}
      \toprule
      & multidivisor & marks &deg. & sing.\\
      \cmidrule{2-5}
      (\ref{item:idx2-2fp})
      &$\{-2,0\} \otimes P_0\ + \   \{\frac{1}{3}\}\otimes P_1\ +\   \{\frac{1}{3}\}\otimes P_2$ &$\QQ^\pm$& $1$ & $K_3A_5A_1$\\
      &$\{-1,0\} \otimes P_0\ + \   \{\frac{1}{4}\}\otimes P_1\ +\   \{\frac{1}{4}\}\otimes P_2$ && $1$ & $K_2A_7$\\
      &$\{-1,0\} \otimes P_0\ + \   \{\frac{2}{5}\}\otimes P_1\ +\   \{\frac{2}{5}\}\otimes P_2$ && $1$ & $K_5A_4$\\
      &$\{-1,0\} \otimes P_0\ + \   \{\frac{1}{6}\}\otimes P_1\ +\   \{\frac{1}{2}\}\otimes P_2$ && $2$ & $K_1A_7$\\
      &$\{-2,-1\} \otimes P_0\ + \   \{\frac{5}{7}\}\otimes P_1\ +\   \{\frac{1}{3}\}\otimes P_2$& & $5$ & $K_5$\\
      &$\{-2,-1\} \otimes P_0\ + \   \{\frac{7}{9}\}\otimes P_1\ +\   \{\frac{1}{3}\}\otimes P_2$& & $2$ & $K_6A_2$\\
      &$\{\frac{-3}{2},-1\} \otimes P_0\ + \   \{\frac{5}{6}\}\otimes P\ +\   \{\frac{1}{2}\}\otimes P_2$& & $1$ & $K_1D_8$\\
      &$\{\frac{-3}{2},\frac{-1}{2}\} \otimes P_0\ + \   \{\frac{1}{2}\}\otimes P\ +\   \{\frac{1}{2}\}\otimes P_2$& & $1$ & $K_12D_4$\\
      &$\{\frac{-4}{3},0\} \otimes P_0\ + \   \{\frac{1}{2}\}\otimes P\ +\   \{\frac{1}{2}\}\otimes P_2$& & $1$ & $K_1D_5A_3$\\
      \cmidrule{2-5}
      (\ref{item:idx2-1fp})&$\{-\frac{3}{4}\}\otimes P_1\ +\   \{\frac{1}{2}\}\otimes P_2\ +\   \{\frac{1}{2}\}\otimes P_3$& $\QQ^+$ &$1$ &$K_1D_62A_1$\\
      \bottomrule
\end{tabular}
\end{center}
\end{spacing}

\begin{thm}
  There are the following log del Pezzo $\CC^*$-surfaces with Picard rank $1$ and index $\ell=3$.
  \begin{enumerate}
  \item  28 surfaces with two elliptic fixed points \label{item:idx3-2fp}
  \item  3  family over $\A^1$ of surfaces with two elliptic fixed points \label{item:idx3-family}
  \item  5  surface with one elliptic fixed point \label{item:idx3-1fp}
  \item  18 toric surfaces (c.f. \cite{Dais:arXiv0709.0999})
  \end{enumerate}
\end{thm}

\begin{spacing}{1.2}
\begin{center}
  \begin{longtable}{lcccc}
    \toprule
    & multidivisor & marks & $(K_X)^2$& sing \\
    \cmidrule{2-5}
    (\ref{item:idx3-2fp})  
    &$\{-4,0\} \otimes P_0\ + \   \{\frac{1}{2}\}\otimes P_1\ +\   \{\frac{1}{2}\}\otimes P_2$ &$\QQ^\pm$ & $\frac{4}{3}$& $A_3(11)2A_3$\\
    &$\{-4,2\} \otimes P_0\ + \   \{\frac{1}{2}\}\otimes P_1\ +\   \{\frac{1}{2}\}\otimes P_2$ && $\frac{2}{3}$&$2A_3(11)A_3$\\
    &$\{-3,0\} \otimes P_0\ + \   \{\frac{1}{4}\}\otimes P_1\ +\   \{\frac{1}{2}\}\otimes P_2$ && $1$&$A_3(12)A_5A_2$\\
    &$\{-2,0\} \otimes P_0\ + \   \{\frac{1}{4}\}\otimes P_1\ +\   \{\frac{1}{4}\}\otimes P_2$ && $\frac{2}{3}$&$A_3(22)A_7A_1$\\
    &$\{-1,0\} \otimes P_0\ + \   \{\frac{1}{5}\}\otimes P_1\ +\   \{\frac{1}{2}\}\otimes P_2$ && $\frac{7}{3}$&$A_1(1)A_6$\\
    &$\{-1,0\} \otimes P_0\ + \   \{\frac{1}{5}\}\otimes P_1\ +\   \{\frac{1}{5}\}\otimes P_2$ && $\frac{2}{3}$&$A_2(22)A_9$\\
    &$\{-1,0\} \otimes P_0\ + \   \{\frac{1}{6}\}\otimes P_1\ +\   \{\frac{1}{3}\}\otimes P_2$ && $1$&$A_2(12)A_8$\\
    &$\{-1,1\} \otimes P_0\ + \   \{\frac{3}{7}\}\otimes P_1\ +\   \{\frac{1}{2}\}\otimes P_2$ && $6$&$A_4(12)A_1$\\
    &$\{-1,0\} \otimes P_0\ + \   \{\frac{3}{7}\}\otimes P_1\ +\   \{\frac{3}{7}\}\otimes P_2$ && $\frac{2}{3}$&$A_5(22)A_6$\\
    &$\{-1,0\} \otimes P_0\ + \   \{\frac{1}{8}\}\otimes P_1\ +\   \{\frac{1}{2}\}\otimes P_2$ && $\frac{5}{3}$&$A_1(2)A_9$\\
    &$\{-1,1\} \otimes P_0\ + \   \{\frac{3}{8}\}\otimes P_1\ +\   \{\frac{1}{2}\}\otimes P_2$ && $\frac{10}{3}$&$A_5(11)2A_1$\\
    &$\{-2,-1\} \otimes P_0\ + \   \{\frac{5}{8}\}\otimes P_1\ +\   \{\frac{2}{5}\}\otimes P_2$ && $\frac{13}{3}$&$A_6(11)$\\
    &$\{-1,1\} \otimes P_0\ + \   \{\frac{3}{10}\}\otimes P_1\ +\   \{\frac{1}{2}\}\otimes P_2$ && $2$&$A_5(12)A_3A_1$\\
    &$\{-2,-1\} \otimes P_0\ + \   \{\frac{7}{10}\}\otimes P_1\ +\   \{\frac{2}{5}\}\otimes P_2$ && $1$&$A_6(12)A_4$\\
    &$\{-2,-1\} \otimes P_0\ + \   \{\frac{10}{13}\}\otimes P_1\ +\   \{\frac{1}{4}\}\otimes P_2$ && $\frac{17}{3}$&$A_6(22)$\\
    &$\{-2,-1\} \otimes P_0\ + \   \{\frac{11}{14}\}\otimes P_1\ +\   \{\frac{1}{4}\}\otimes P_2$ && $3$&$A_7(12)A_1$\\
    &$\{-2,-1\} \otimes P_0\ + \   \{\frac{13}{16}\}\otimes P_1\ +\   \{\frac{1}{4}\}\otimes P_2$ && $\frac{5}{3}$&$A_7(22)A_3$\\
    &$\{\frac{-3}{2},0\} \otimes P_0\ + \   \{\frac{1}{2}\}\otimes P_1\ +\   \{\frac{1}{2}\}\otimes P_2$ && $\frac{4}{3}$&$A_1(1)D_4A_3$\\
    &$\{\frac{-3}{2},0\} \otimes P_0\ + \   \{\frac{1}{4}\}\otimes P_1\ +\   \{\frac{1}{2}\}\otimes P_2$ && $\frac{2}{3}$&$D_4(1)A_1(1)A_5$\\
    &$\{\frac{-3}{2},-1\} \otimes P_0\ + \   \{\frac{4}{5}\}\otimes P_1\ +\   \{\frac{1}{2}\}\otimes P_2$ && $\frac{4}{3}$&$A_1(1)D_7$\\
    &$\{\frac{-3}{2},-1\} \otimes P_0\ + \   \{\frac{4}{7}\}\otimes P_1\ +\   \{\frac{1}{2}\}\otimes P_2$ && $\frac{17}{3}$&$\mathbf{D_5(2)}$\\
    &$\{\frac{-3}{2},-1\} \otimes P_0\ + \   \{\frac{5}{8}\}\otimes P_1\ +\   \{\frac{1}{2}\}\otimes P_2$ && $\frac{8}{3}$&$\mathbf{D_6(1)A_1}$\\
    &$\{\frac{-3}{2},-1\} \otimes P_0\ + \   \{\frac{7}{8}\}\otimes P_1\ +\   \{\frac{1}{2}\}\otimes P_2$ & & $\frac{2}{3}$&$A_1(2)D_{10}$\\
    &$\{\frac{-3}{2},-1\} \otimes P_0\ + \   \{\frac{7}{10}\}\otimes P_1\ +\   \{\frac{1}{2}\}\otimes P_2$ & & $\frac{4}{3}$&$D_6(2)A_3$\\
    &$\{\frac{-4}{3},-1\} \otimes P_0\ + \   \{\frac{4}{5}\}\otimes P_1\ +\   \{\frac{1}{2}\}\otimes P_2$ & & $\frac{1}{3}$&$A_1(1)E_8$\\
    &$\{\frac{-4}{3},\frac{-2}{3}\} \otimes P_0\ + \   \{\frac{1}{2}\}\otimes P_1\ +\   \{\frac{1}{2}\}\otimes P_2$ & & $\frac{2}{3}$&$A_1(2)2D_5$\\
    &$\{\frac{-5}{4},\frac{-1}{2}\} \otimes P_0\ + \   \{\frac{1}{2}\}\otimes P_1\ +\   \{\frac{1}{2}\}\otimes P_2$ & & $\frac{2}{3}$&$A_1(2)D_4D_6$\\
    &$\{\frac{-6}{5},0\} \otimes P_0\ + \   \{\frac{1}{2}\}\otimes P_1\ +\   \{\frac{1}{2}\}\otimes P$ & &  $\frac{2}{3}$&$A_1(2)D_7D_3$\\
   \cmidrule{2-5}
    (\ref{item:idx3-family}) 
    &$\{-3,-1\} \otimes P_0+  \{\frac{1}{2}\}\otimes P_1 +  \{\frac{1}{2}\}\otimes P_2 +  \{\frac{1}{2}\}\otimes P_3$ &$\QQ^\pm$ &  $\frac{2}{3}$&$D_4(1)D_4A_1$\\
    &$\{-3,0\} \otimes P_0 + \{\frac{1}{2}\}\otimes P_1 +  \{\frac{1}{2}\}\otimes P_2 +  \{\frac{1}{2}\}\otimes P_3$ && $\frac{1}{3}$&$2D_4(1)A_2$\\
    &$\{-2,-1\} \otimes P_0+  \{\frac{1}{4}\}\otimes P_1 +  \{\frac{1}{2}\}\otimes P_2 +  \{\frac{1}{2}\}\otimes P_3$ && $\frac{1}{3}$&$D_4(2)D_6$\\
   \cmidrule{2-5}
   (\ref{item:idx3-1fp})  
   &${\{\frac{1}{2}\}\otimes P_1\ +\   \{\frac{1}{2}\}\otimes P_2\ +\  \{\frac{1}{2}\}\otimes P_3}$ &$\QQ^+$& $\frac{8}{3}$&$\mathbf{D_4(1)3A_1}$\\
   &$\{-\frac{2}{3}\}\otimes P_1\ +\   \{\frac{1}{2}\}\otimes P_2\ +\   \{\frac{1}{2}\}\otimes P_3$ && $\frac{4}{3}$&$A_1(1)D_52A_1$\\
   &$\{-\frac{2}{3}\}\otimes P_1\ +\   \{\frac{1}{3}\}\otimes P_2\ +\   \{\frac{1}{2}\}\otimes P_3$ && $\frac{2}{3}$&$2A_1(1)E_6A_1$\\
   &$\{-\frac{1}{4}\}\otimes P_1\ +\   \{\frac{1}{2}\}\otimes P_2\ +\   \{\frac{1}{2}\}\otimes P_3$ && $\frac{4}{3}$&$D_4(2)A_32A_1$\\
   &$\{-\frac{5}{6}\}\otimes P_1\ +\   \{\frac{1}{2}\}\otimes P_2\ +\   \{\frac{1}{2}\}\otimes P_3$ && $\frac{2}{3}$&$A_1(2)D_82A_1$\\
   \bottomrule
  \end{longtable}
\end{center}
\end{spacing}

\begin{rem}
  In order to compare our results with that of Hisanori Ohashi and Shingo Taki in \cite{ldp3k3} we used in the list their notation for the quotient singularities of Gorenstein index $3$:
 \[
 \begin{array}{ll}
    A_1(1) = \langle 3,1\rangle
   &A_1(2) = \langle 6,1\rangle\\
   A_m(11) = \langle 9m - 15, 6m-11\rangle&
   A_m(12) = \langle 9m - 9, 6m-7\rangle\\
   &A_m(22) = \langle 9m - 3, 3m-2\rangle\\
   \\
   D_m(1) = \langle 2; 2,1;2,1;3m-10,3m-13\rangle&
   D_m(2) = \langle 2; 2,1;2,1;3m-8,3m-11\rangle
 \end{array}
 \]
 In their paper Hisanori Ohashi and Shingo Taki classified log del Pezzo surfaces index $3$ fulfilling the following \emph{smooth divisor property}:
$|-3K_X|$ has to contain a divisor $2C$, such that $C$ is a smooth curve which does not meet the singularities of $X$. 

Comparing the results one finds that all singularity types for Picard rank $1$ in their list are either realized by toric surfaces or by $k^*$-surfaces from our list
Moreover, also for Picard rank two we may realize some of the non-toric configurations mentioned there by $k^*$ surfaces:
\begin{spacing}{1.2}
\begin{longtable}{ccc}
  \toprule
   multidivisor & marks & sing \\
   $\{\frac{-1}{2},0\} \otimes P_0\ + \   \{-\frac{1}{2},0\}\otimes P_1\ +\   \{\frac{1}{4}\}\otimes P_2$ &$\QQ^\pm$& $D_4(2)$\\
$\{-\frac{1}{2},0\} \otimes P_0\ + \   \{-\frac{1}{2},0\}\otimes P_1\ +\   \{\frac{2}{5}\}\otimes P_2$ && $D_5(1)A_1$\\
\midrule
$\{\frac{1}{2},0\} \otimes P_0\ + \   \{-\frac{1}{2}\}\otimes P_1\ +\   \{\frac{1}{2}\}\otimes P_2$ &$\QQ^+$& $D_4(1)2A_1$\\
\end{longtable}
\end{spacing}
\end{rem}
From our results and that of \cite{MR1009466} we know that there exist families of log Del Pezzo surfaces of Picard rank $1$ and index $\ell=1,3$ but not of index $2$. So one may ask oneself for which indices there are such families.
A partial answer is given by the following 
\begin{prop}
  For every odd $\ell \in \NN$ there exists a family of log del Pezzo surfaces of Picard rank $1$ and Gorenstein index $\ell$. 
\end{prop}

\begin{proof}
  For $\ell = 2i +1$. The multidivisors with marks $\{\QQ^-,\QQ^+\}$
  \begin{enumerate}
  \item 
    $\fan_{i}:=\{-2-i,-1\} \otimes P_1 + \frac{1}{2} \otimes P_2 + \frac{1}{2} \otimes P_3 + \frac{1}{2} \otimes P_4$
  \item $\fan_i':=\{-2-i,i-1\} \otimes P_1 + \frac{1}{2} \otimes P_2 + \frac{1}{2} \otimes P_3 + \frac{1}{2} \otimes P_4$
  \end{enumerate}
   give rise to families over $\PP^1 \setminus \{0,1, \infty\}$.
\end{proof}

\begin{rem}
  We may also describe the total space of these families by the corresponding divisorial fans.
  $$\fan_i = \{\D_-, \D, \D_+\}$$
  with 
  \begin{eqnarray*}
    \D_-&=& (-\infty,-2-i] \otimes D_0+ 
    (-\infty,\frac{1}{2}] \otimes D_1+
    (-\infty,\frac{1}{2}] \otimes D_\infty+
    (-\infty,\frac{1}{2}] \otimes D_\Delta\\
    \D_+&=& [-1 ,\infty) \otimes D_0+ 
    [\frac{1}{2},\infty) \otimes D_1+
    [\frac{1}{2},\infty) \otimes D_\infty+
    [\frac{1}{2},\infty)  \otimes D_\Delta\\
    \D&=& [-2-i,-1] \otimes D_0+ 
    \emptyset \otimes D_1+
    \emptyset \otimes D_\infty+
    \emptyset \otimes D_\Delta\\
  \end{eqnarray*}
  with $D_1=\{1\}\times \PP^1, D_0=\{0\}\times \PP^1, D_\infty=\{\infty\}\times \PP^1$ and the diagonal $D_\Delta = \Delta \subset \PP^1 \times \PP^1$.
\end{rem}

\section{Examples in dimension 3}
\label{sec:examples-dimension-3}
\begin{exmp}[Two smoothings]
  We consider $X=\text{Cone}(\text{dP}_6)$ the projective cone over the del Pezzo surface of degree 6. In \cite{jahnke-2006} it serves as an example of a Fano threefold with canonical singularities admitting two different smoothings. It turns out, that both smoothings are equivariant. Hence, they can be described by their divisorial fans.
  \begin{figure}[htbp]
     \subfigure[$\fan_0$]{
       \includegraphics[trim=0 0 0 0,clip=true, height=4cm]{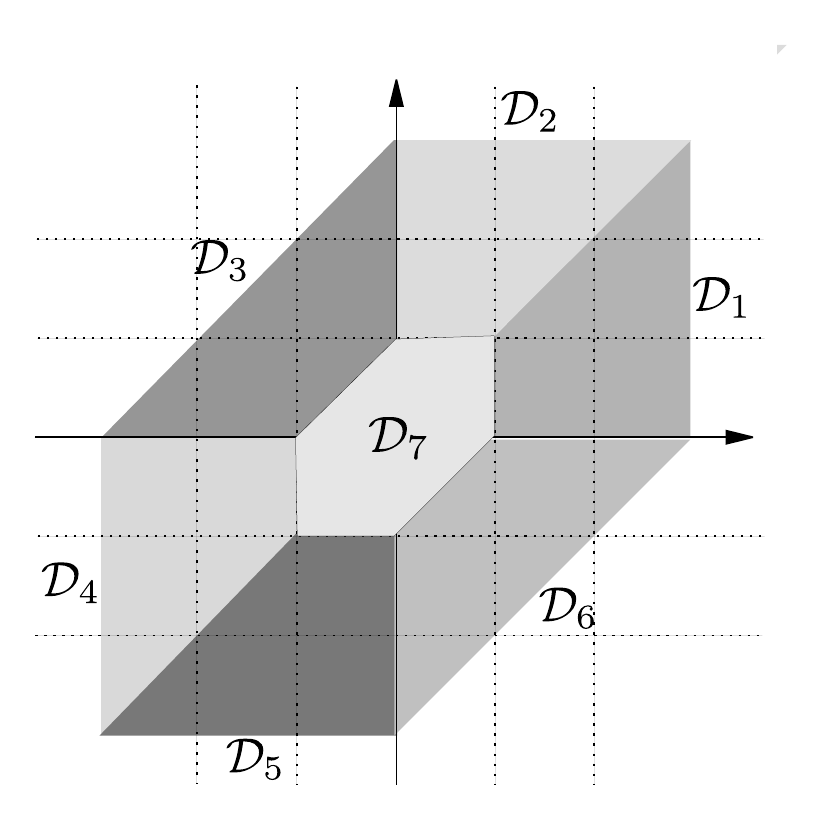}%
    }
    \subfigure[$\fan_1$]{
       \includegraphics[trim=0 0 0 0,clip=true, height=4cm]{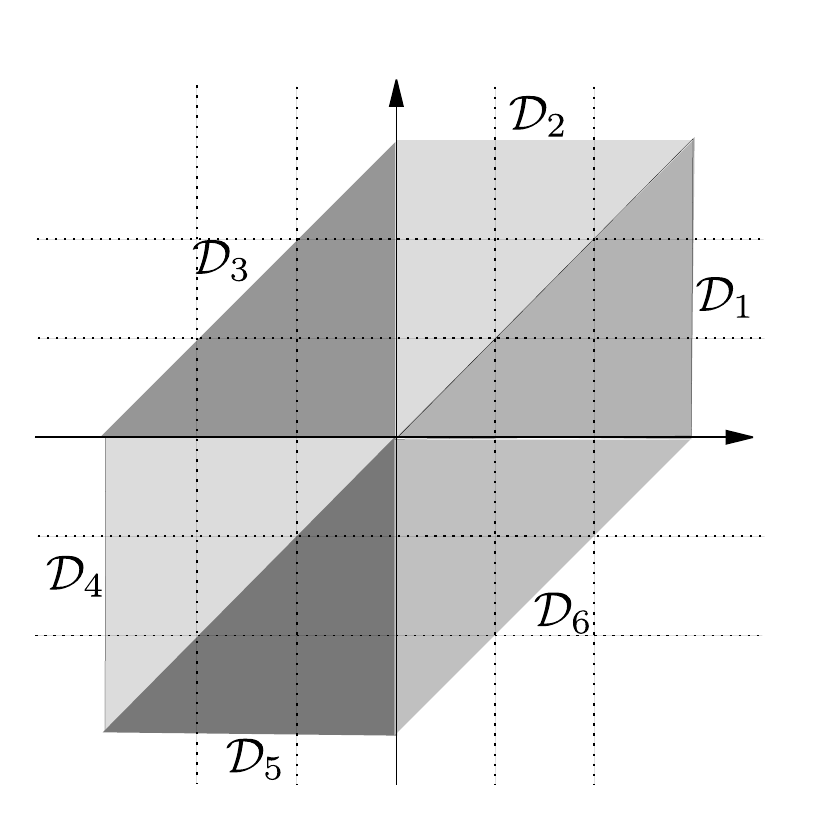}%
    }
    \caption{Projective cone over the del Pezzo surface of degree 6}
    \label{fig:dp6}
  \end{figure}

The projective cone over $\text{dP}_6$ is toric. By restricting the torus action we get a divisorial fan over $Y=\PP^1$ with two nontrivial slices---see figure~\ref{fig:dp6}. Note that $\D_7$ does not occur in the slice $\fan_1$, this means it has coefficient $\emptyset$ at $\{1\}$.

Now we consider divisorial fans over $Y=\PP^1 \times \mathbf{A}^1$
consisting of seven divisors. Their non-trivial coefficients at the prime divisors are shown in figure~\ref{fig:smoothing-tp2} and \ref{fig:smoothing-ppp}, respectively. Here we use the coordinates $x=\frac{u}{v}$ and $y$ for $(u:v,y) \in \PP^1 \times \mathbf{A}^1$.

  \begin{figure}[htbp]
    \subfigure[$\fan_{\{y=0\}}$]{
      \includegraphics[trim=5 0 5 17,clip=true, height=2.9cm]{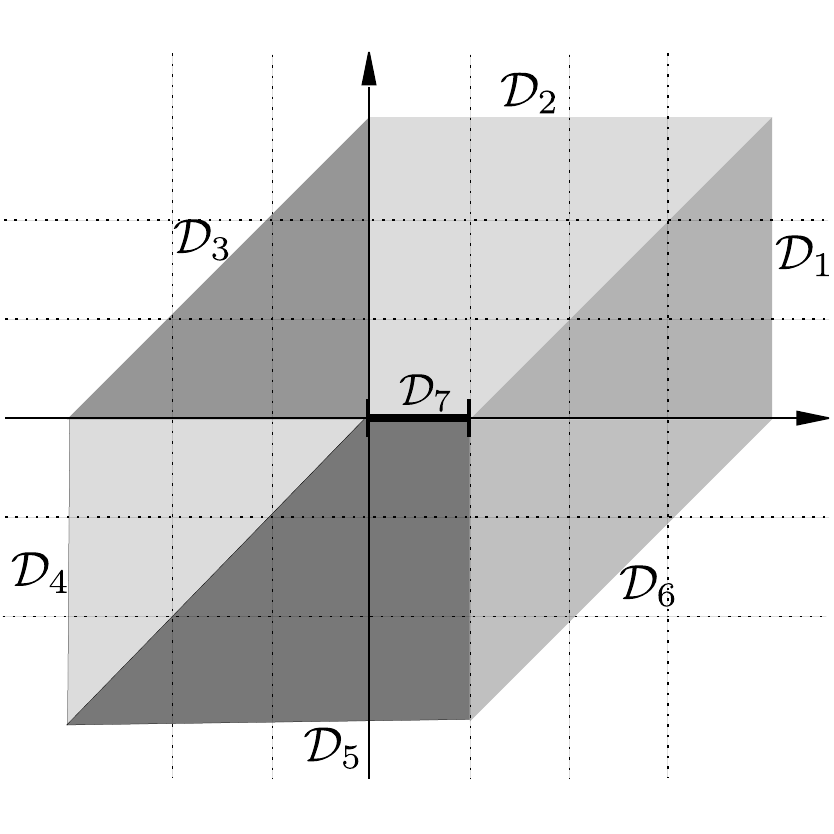}%
    }
    \subfigure[$\fan_{\{y=x\}}$]{
      \includegraphics[trim=5 18 5 0,clip=true,height=2.8cm]{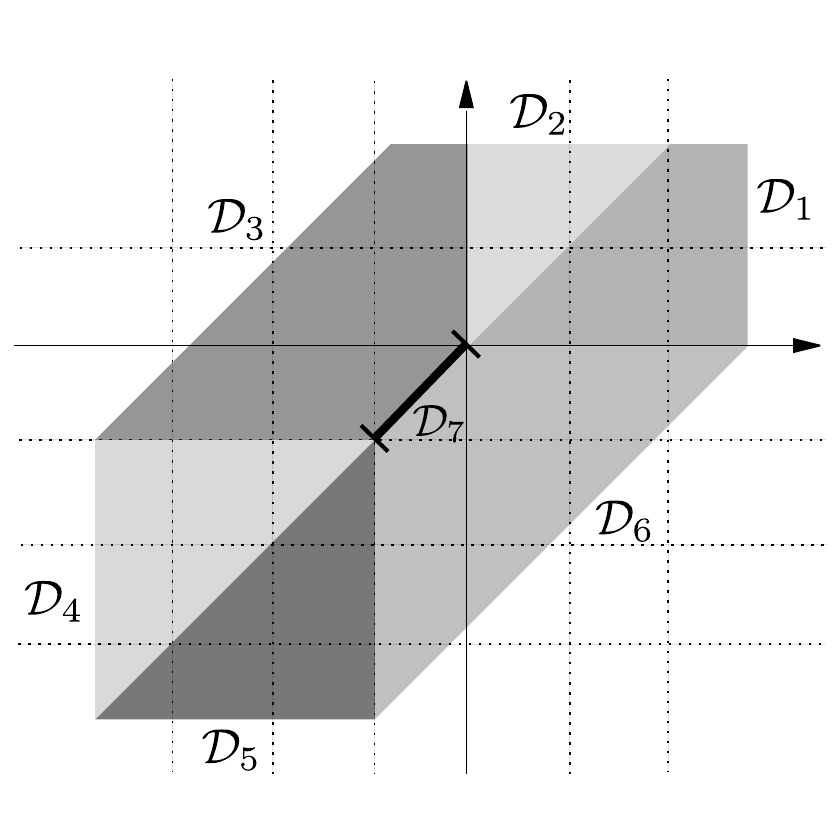}%
    }
    \subfigure[$\fan_{\{y=2x\}}$]{
      \includegraphics[trim=5 0 5 23,clip=true,height=2.9cm]{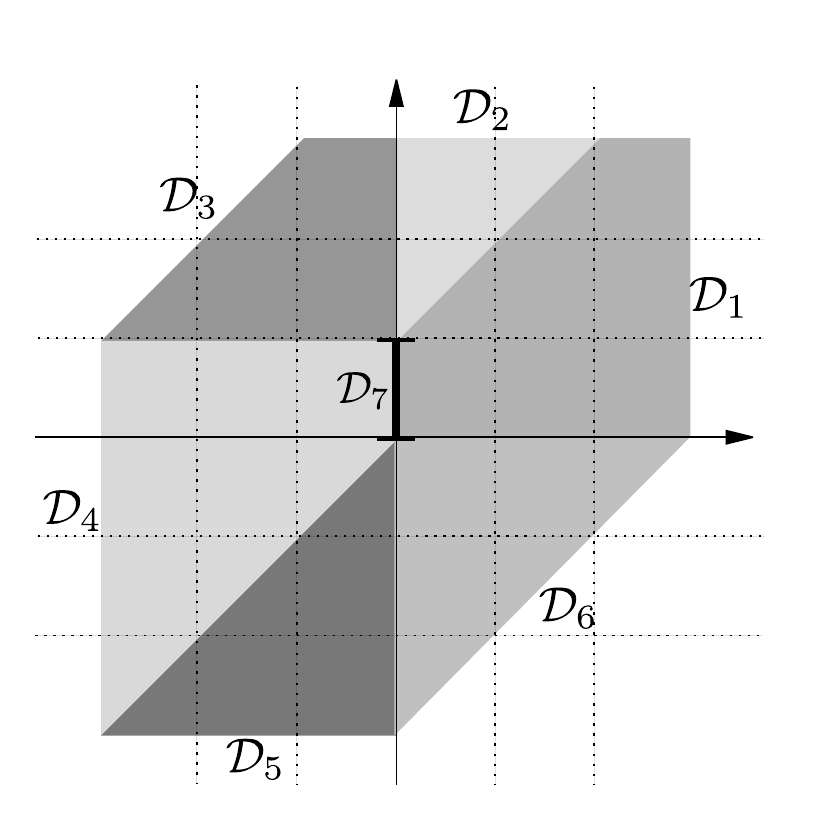}%
    }
     \subfigure[$\fan_{\{y=1\}}$]{
       \includegraphics[trim=15 0 10 23,clip=true, height=2.9cm]{cone0.pdf}%
    }
   \caption{Smoothing to $\PP(T_{\PP^2})$}
\label{fig:smoothing-tp2}
\end{figure}

Note that $k[y] \subset \Gamma(\pdv(\D_i),\CO_{\pdv(\D_i)})_0$ for $i=1,\ldots 7$. So we get dominant and torus invariant morphisms $\pdv(\D_i)\rightarrow \mathbf{A}^1$. They glue to a morphism $\pdv(\fan) \rightarrow \mathbf{A}^1$.

\begin{figure}[htbp]
    \subfigure[$\fan_{\{y=0\}}$]{
      \includegraphics[trim=0 0 0 0,clip=true, height=4cm]{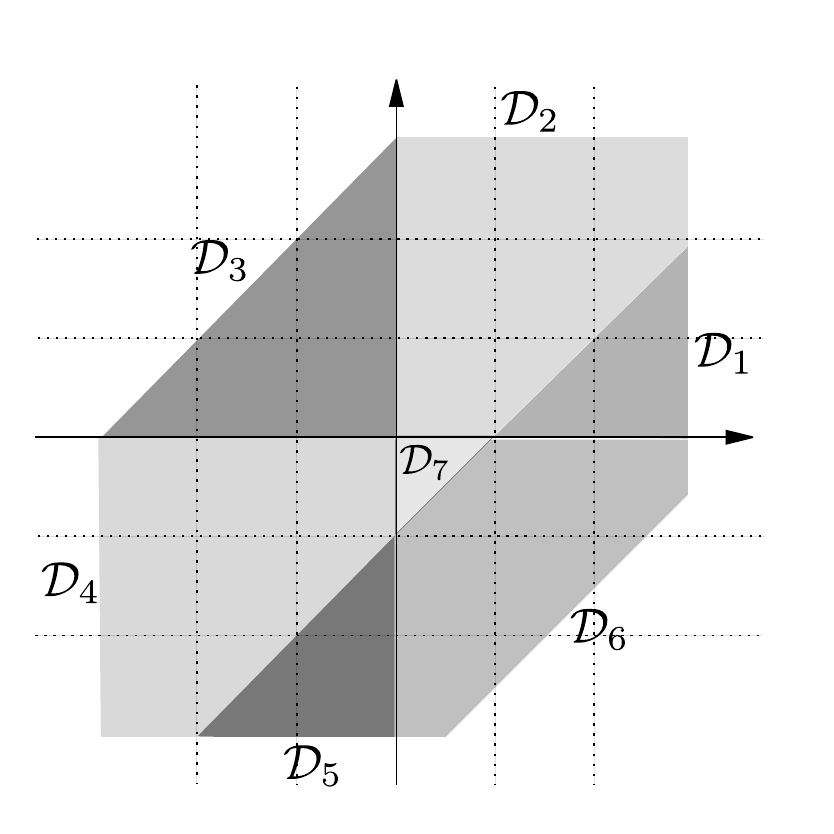}%
    }
    \subfigure[$\fan_{\{y=x\}}$]{
      \includegraphics[trim=0 0 0 0,clip=true,height=4cm]{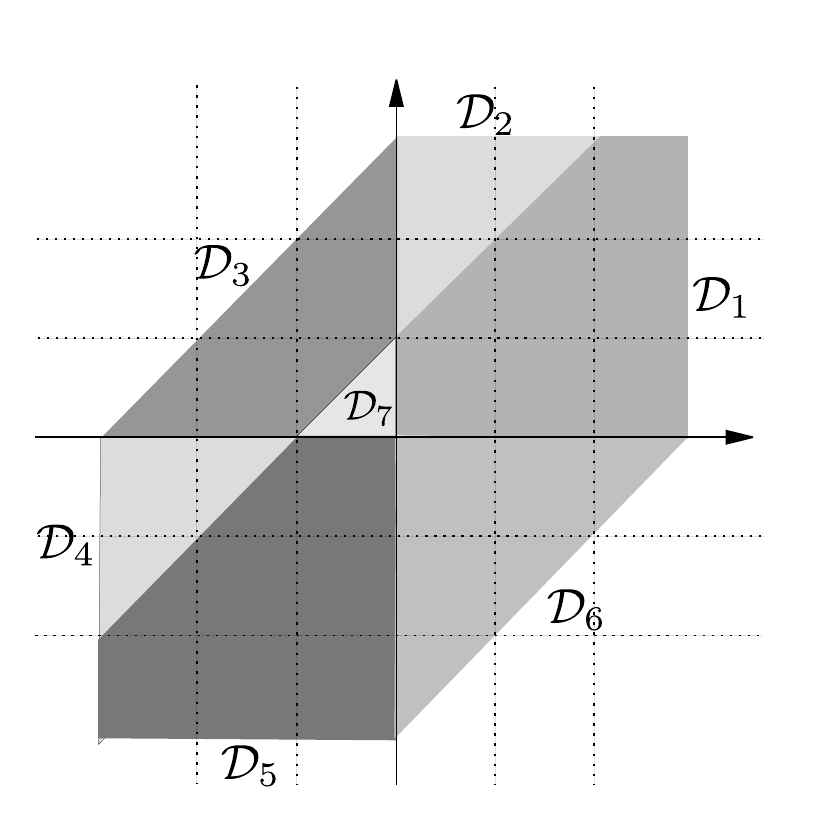}%
    }
     \subfigure[$\fan_{\{y=1\}}$]{
       \includegraphics[trim=0 0 0 0,clip=true, height=4cm]{cone0.pdf}%
     }
     \caption{Smoothing to $\PP^1 \times \PP^1 \times \PP^1$}
\label{fig:smoothing-ppp}
  \end{figure}
It is easy to check that the generic fibre $f_y$ of these 
families is given by divisorial fans over $\PP^1$ with the same slices but at $0,1,y$ . The fibre  $f_0$ is $\text{Cone}(\text{dP}_6)$.

Note that the equivariant smoothing leads to a Minkowski decomposition of the polyhedron corresponding to the singular fixed point. This effect already occurred in the context of deformations of toric singularities in \cite{MR1452429}. 
But note, that here we also give a description of the fibres in terms of polyhedral divisors.

As an object of future studies one may ask when such equivariant smoothings exist and if they are unique.
As one expects from the examples this question is closely related to the existence and uniqueness of {\em smooth} Minkowski decompositions.
\end{exmp}

As a last application we state the following 
\begin{prop}
  $\PP^3$ and $Q^3$ are the only smooth threefolds of Picard rank $1$ admitting a $T^2$-action.
\end{prop}
\begin{proof}
  There exists exactly one smooth toric threefold of Picard rank one, namely $\PP^3$.

  To find  divisorial fans describing non-toric threefolds first note, that $Y$ has to be $\PP^1$, i.e. $X$ is rational. Indeed, there are at least three rays in the tail fan. Because of the formula for the picard rank
  $$\rho_X= 1 + \# \xrays(\fan) + \sum_{P \in Y} (\# \fan_P^{(0)} - 1) - \dim N.$$
 from theorem~\ref{sec:thm-picard-rank} we know that there must be at least one non-extremal ray, so there is a divisor with complete locus $\D \in \fan$. Since $\pdv(\D)$ has to be non-singular, by theorem~\ref{sec:prop-D-smooth} we get $Y=\PP^1$.
  
Because of the statement on isomorphy in theorem~\ref{sec:thm-fan-iso-codim1} we may assume that no slice is a translation of the tail fan by a lattice element.

 $\pdv(\fan)$ being non-toric implies that there are at least three slices which are non-trivial.
  Consider $\D \in \fan$ with complete locus and tail cone $\sigma \subset \tail \fan$. We must have a slice with at least two vertices. ideed, if we had only translated tail fans $\D$ would have at least three non-integral coefficients this contradicts the smoothness of $\pdv(\D)$ by theorem~\ref{sec:prop-D-smooth}.
By the rank formula we have at most one extremal ray and therefore at most one divisor with non-compact tail and affine locus.
 Let $\fan_{P_1}$ have at least two vertices then there are two non-compact polyhedra which have two vertices. One of them belongs to a divisor with complete support. Hence, there is at most one further non-trivial coefficient of $\D$, i.e. at most on tail fan translated by a non-integral value, so there must be another slice $\fan_{P_2}$ with at least two vertices.

Now the Picard rank formula implies, that we have exactly two slices with exactly two vertices and no extremal rays, i.e. all non-compact polyhedra in the slices belong to divisors with complete locus, so proposition~\ref{sec:prop-D-smooth} implies that the vertices of $\fan_{P_1},\fan_{P_2}$ are integral and there is exactly one slice $\fan_Q$ which is a non-integral translation of tail fan.

Let $\overline{v_1w_1}$ and $\overline{v_2w_2}$ be the line segments in $\fan_{P_1}$ and $\fan_{P_2}$, respectively and $\frac{1}{q}u$ the non-integral vertices in $\fan_Q$ (here $q\in \NN$ minimal such that $u \in N$). Since there are no extremal rays. On all rays of $\tail \fan$ there must lie a vertex of $\Box=\overline{v_1w_1} + \overline{v_2w_2} + r$. So $\fan$ has a tail fan with four rays which are spanned by the vertices of $\Box$. We have four polyhedral divisors with complete locus and tail cones of maximal dimension. 
\begin{eqnarray*}
  \D_1=(\overline{v_1w_1}+\sigma_1)\otimes P_1 + (v_2+\sigma_1)\otimes P_2 + (\frac{1}{q}u + \sigma_1)\otimes Q\\
  \D_2=(\overline{v_2w_2}+\sigma_2)\otimes P_1 + (v_1+\sigma_2)\otimes P_2 + (\frac{1}{q}u + \sigma_2) \otimes Q\\
  \D_3=(\overline{v_1w_1}+\sigma_3)\otimes P_1 + (w_2+\sigma_3)\otimes P_2 + (\frac{1}{q}u + \sigma_3) \otimes Q\\
  \D_4=(\overline{v_2w_2}+\sigma_4)\otimes P_1 + (w_1+\sigma_4)\otimes P_2 + (\frac{1}{q}u + \sigma_4) \otimes Q
\end{eqnarray*}

By remark~\ref{sec:rem-toric} and corollary~\ref{sec:cor-toric} we see that $\pdv(\D_i)$ is essentially toric 
with $\pdv(\D_i) \cong X_{\delta_i}$ with 
\begin{eqnarray*}
  \delta_1=\pos(
  \left(\begin{smallmatrix}
    v_1+v_2 \\ 1
  \end{smallmatrix}\right),
  \left(\begin{smallmatrix}
    w_1+v_2 \\ 1
  \end{smallmatrix}\right)
  \left(\begin{smallmatrix}
    u \\ -q
  \end{smallmatrix}\right)
)\\
\delta_2=\pos(
  \left(\begin{smallmatrix}
    v_2+v_1 \\ 1
  \end{smallmatrix}\right),
  \left(\begin{smallmatrix}
    w_2+v_1 \\ 1
  \end{smallmatrix}\right)
  \left(\begin{smallmatrix}
    u \\ -q
  \end{smallmatrix}\right)
)\\
\delta_3=\pos(
  \left(\begin{smallmatrix}
    v_1+w_2 \\ 1
  \end{smallmatrix}\right),
  \left(\begin{smallmatrix}
    w_1+w_2 \\ 1
  \end{smallmatrix}\right)
  \left(\begin{smallmatrix}
    u \\ -q
  \end{smallmatrix}\right)
)\\
\delta_4=\pos(
  \left(\begin{smallmatrix}
    v_2+w_1 \\ 1
  \end{smallmatrix}\right),
  \left(\begin{smallmatrix}
    w_2+w_1 \\ 1
  \end{smallmatrix}\right)
  \left(\begin{smallmatrix}
    u \\ -q
  \end{smallmatrix}\right)
).
\end{eqnarray*}

We get
\begin{eqnarray*}
 4  &=& \left|\det(
\left(\begin{smallmatrix}
    v_1+v_2 & w_1+v_2 & u\\ 
    1       &   1     & -q
\end{smallmatrix}\right)\right| + 
\left|\det(
\left(\begin{smallmatrix}
    v_2+v_1 & w_2+v_1 & u\\ 
    1       &   1     & -q
\end{smallmatrix}\right)\right| \\
&+&
\left|\det(
\left(\begin{smallmatrix}
    v_1+w_2 & w_1+w_2 & u\\ 
    1       &   1     & -q
\end{smallmatrix}\right)\right|+
\left|\det(
\left(\begin{smallmatrix}
    v_2+w_1 & w_2+w_1 & u\\ 
    1       &   1     & -q
\end{smallmatrix}\right)\right|\\
&=& 2q \left|\det(
\left(\begin{smallmatrix}
    v_1-w_1 & v_2-w_2
\end{smallmatrix}\right)\right| = 2q \vol(\Box)
\end{eqnarray*}
Because $q>1$ this implies, that $q=2$ and $\vol(\Box)=1$. So after applying a lattice automorphism we get the slices of figure~\ref{fig:quadric}.
\end{proof}

\bibliographystyle{alpha}
\bibliography{tfano.bib}

\end{document}